\documentclass[reqno,12pt]{amsart}
\usepackage{color}
\usepackage{amsthm}
\usepackage{amsmath}
\usepackage{epsfig}
\usepackage{graphicx}
\usepackage{amssymb}
\usepackage{latexsym}
\usepackage{pdfpages}

\usepackage{enumitem}

\usepackage{tikz}

\usetikzlibrary{babel}
\usetikzlibrary{calc}

\usepackage{ulem} 

\usepackage{ifpdf}
\newcommand{\res}{\!\!\mathop{\hbox{
                                \vrule height 7pt width .5pt depth 0pt
                                \vrule height .5pt width 6pt depth 0pt}}
                                \nolimits}

\def\z{{\bf z}}

\ifpdf 
  \usepackage[hidelinks]{hyperref}
\else 
\fi 

 \usepackage[usenames,dvipsnames]{pstricks}
 \usepackage{epsfig}
 \usepackage{pst-grad} 
 \usepackage{pst-plot} 

\usepackage{color}

\newtheorem{theorem}{Theorem}[section]
\newtheorem{lemma}[theorem]{Lemma}
\newtheorem{definition}[theorem]{Definition}
\newtheorem{proposition}[theorem]{Proposition}
\newtheorem{corollary}[theorem]{Corollary}
\newtheorem{remark}[theorem]{Remark}
\newtheorem{example}[theorem]{Example}
\newtheorem*{theorem*}{\it Theorem}

\def\vint_#1{\mathchoice%
          {\mathop{\kern 0.2em\vrule width 0.6em height 0.69678ex depth -0.58065ex
                  \kern -0.8em \intop}\nolimits_{\kern -0.4em#1}}%
          {\mathop{\kern 0.1em\vrule width 0.5em height 0.69678ex depth -0.60387ex
                  \kern -0.6em \intop}\nolimits_{#1}}%
          {\mathop{\kern 0.1em\vrule width 0.5em height 0.69678ex
              depth -0.60387ex
                  \kern -0.6em \intop}\nolimits_{#1}}%
          {\mathop{\kern 0.1em\vrule width 0.5em height 0.69678ex depth -0.60387ex
                  \kern -0.6em \intop}\nolimits_{#1}}}
\def\vintslides_#1{\mathchoice%
          {\mathop{\kern 0.1em\vrule width 0.5em height 0.697ex depth -0.581ex
                  \kern -0.6em \intop}\nolimits_{\kern -0.4em#1}}%
          {\mathop{\kern 0.1em\vrule width 0.3em height 0.697ex depth -0.604ex
                  \kern -0.4em \intop}\nolimits_{#1}}%
          {\mathop{\kern 0.1em\vrule width 0.3em height 0.697ex depth -0.604ex
                  \kern -0.4em \intop}\nolimits_{#1}}%
          {\mathop{\kern 0.1em\vrule width 0.3em height 0.697ex depth -0.604ex
                  \kern -0.4em \intop}\nolimits_{#1}}}

\def\R{\mathbb R}
\def\N{\mathbb N}

\numberwithin{equation}{section}

\def\1{\raisebox{2pt}{\rm{$\chi$}}}
\def\e{{\bf e}}

\def\v{{\bf v}}
\def\b{{\bf b}}

\renewcommand{\v}{\mathrm{v}}
\newcommand{\vf}{\mathrm{f}}
\newcommand{\vi}{\mathrm{i}}
\newcommand{\V}{\mathrm{V}}

\newcommand{\E}{\mathrm{E}}

\numberwithin{equation}{section}

\def\a{\mathbf{a}}
\def\1{\raisebox{2pt}{\rm{$\chi$}}}

\definecolor{violet(ryb)}{rgb}{0.53, 0.0, 0.69}
\definecolor{internationalorange}{rgb}{1.0, 0.31, 0.0}
\usepackage{mathtools}
\mathtoolsset{showonlyrefs}

\begin{document}

\title[The Cheeger Problem]{\bf The Cheeger Cut and Cheeger Problem in Metric Graphs }

\author[J. M. Maz\'on]{Jos\'e M. Maz\'on}

\address{J. M. Maz\'{o}n: Departamento de An\'{a}lisis Matem\'{a}tico,
Universitat de Val\`encia, Dr. Moliner 50, 46100 Burjassot, Spain.
 {\tt mazon@uv.es }}

\keywords{Cheeger problem, Cheeger cut, metric graphs, functions of total variation, total variation flow, the $1$-Laplacian\\
\indent 2010 {\it Mathematics Subject Classification:5R02,05C21, 47J35}
}

\setcounter{tocdepth}{1}

\date{\today}

\begin{abstract}
 For discrete  weighted graphs there  is sufficient literature about the Cheeger cut and the Cheeger problem, but  for metric graphs there are few results about
   these problems.
Our aim is to study the Cheeger cut and the Cheeger problem in metric graphs. For that, we use the concept of total variation and perimeter in metric graphs introduced in \cite{Mazon}, which  takes into account the jumps at the vertices of the functions of bounded variation. Moreover, we study the eigenvalue problem for the minus $1$-Laplacian operator in metric  graphs, whereby we give a method to solve the  optimal Cheeger cut problem.
\end{abstract}

\maketitle

{ \renewcommand\contentsname{Contents}
\setcounter{tocdepth}{3}
\addtolength{\parskip}{-0.2cm}
{\small \tableofcontents}
\addtolength{\parskip}{0.2cm} }

\section{Introduction}

A metric graph is a combinatorial graph where the edges are considered as intervals of the real line with a distance on each one of them and are glued together according to the combinatorial structure. The resulting metric measure space allows to introduce a family of differential operators acting on each edge $\e$ considered as  an interval $(0, \ell_\e)$ with boundary conditions at the vertices. We refer to the pair  formed by the metric graph and the  family of differential operators as {\it quantum graph}. During the last two decades, quantum graphs became an extremely popular
subject because of numerous applications in mathematical physics, chemistry
and engineering. Indeed, the literature on quantum graphs is vast and extensive and there is no chance to give even a brief overview of the subject here.
We only mention a few recent monographs and collected works with a comprehensive bibliography \cite{BCFK}, \cite{BK}, \cite{EKKST}, \cite{GS}, \cite{KS}, \cite{Mugnolo} and \cite{Post}.

The historical motivation of the Cheeger cut problem is an isoperimetric-type inequality
that was first proved by J. Cheeger in \cite{Cheeger}  in the context of compact, $n$-dimensional
Riemannian manifolds without boundary. As a consequence, one obtains the validity
of a Poincar\'{e} inequality with optimal constant uniformly bounded from below by
a geometric constant. Let $\lambda_1(M)$ be the least non-zero eigenvalue of the Laplace-Beltrami operator on $M$, then Cheeger proved that
\begin{equation}\label{Cheeger1}
\lambda_1(M)\geq \frac12 h(M)^2, \quad h(M):= \inf_{A \subset M} \frac{P(A)}{ \min\{ V(A), V(M \setminus A) \}}
\end{equation}
where $V (A)$ and $P(A)$ denote, respectively, the Riemannian volume and perimeter
of $A$.

 The first  Cheeger estimates on discrete graphs are due to Dodziuk \cite{D}
and Alon and Milmann \cite{AM}. Since then, these estimates have been improved and various variants have been proved. Consider  a finite weighted connected graph $G =(V, E)$, where $V = \{x_1, \ldots , x_n \}$ is the set of vertices  (or nodes) and $E$ the set of edges, which are weighted by a function  $w_{ji}= w_{ij} \geq 0$, $(x_i,x_j) \in E$. In this context, the Cheeger cut value of a partition $\{ S, S^c\}$ ($S^c:= V \setminus S$) of $V$ is defined as
$$\mathcal{C}(S):= \frac{{\rm Cut}(S,S^c)}{\min\{{\rm vol}(S), {\rm vol}(S^c)\}},$$
where
${\rm Cut}(A,B) = \sum_{x_i \in A, x_j \in B} w_{ij}$
and ${\rm vol}(S)$ is the volume of $S$, defined as ${\rm vol}(S):= \sum_{x_i \in S} d_{x_i}$, being $d_{x_i}:= \sum_{j=1}^n w_{x_i,x_j}$ the weight at the vertex $x_i$.  Then,
\begin{equation}\label{NPCheeger}h(G) := \min_{S \subset V} \mathcal{C}(S)\end{equation}
is called the {\it Cheeger constant}, and a partition $\{ S, S^c\}$ of $V$ is called a {\it Cheeger cut} of $G$ if $h(G)=\mathcal{C}(S)$. Unfortunately, the Cheeger minimization problem of computing $h(G)$ is NP-hard (\cite{HB}, \cite{SB1}). However, it turns out that $h(G)$ can be approximated by the first positive eigenvalue $\lambda_1$ of  the graph Laplacian  thanks to the following Cheeger inequality (\cite{Ch}):
$$
\frac{\lambda_1}{2} \leq h(G) \leq \sqrt{2\lambda_1}.
$$

 This motivates the spectral clustering method (\cite{Luxburg}), which, in its simplest form, thresholds the least non-zero eigenvalue of the graph Laplacian to get an approximation to the Cheeger constant and, moreover, to a Cheeger cut. In order to achieve a better approximation than the one provided by the classical spectral clustering method, a spectral clustering based on the graph $p$-Laplacian was developed in \cite{BH1}, where it is showed that the second eigenvalue of the graph $p$-Laplacian  tends to the Cheeger constant $h(G)$ as $p \to 1^+$. In  \cite{SB1} the idea was further developed by directly considering the variational characterization of the Cheeger constant $h(G)$
 \begin{equation}\label{cheegerSB}
 h(G) = \min_{u \in L^1} \frac{ \vert u \vert_{TV}}{\Vert u - {\rm median}(u)) \Vert_1},
 \end{equation}
where
$$\vert u \vert_{TV} := \frac{1}{2} \sum_{i,j=1}^n w_{ij} \vert u(x_i) - u(x_j) \vert.$$

 In \cite{SB1}, it was proved that  the solution of the variational problem \eqref{cheegerSB}  provides an exact solution of the Cheeger cut problem. If a global minimizer $u$ of \eqref{cheegerSB} can be computed, then it can be shown that this minimizer would be the indicator function of a set $\Omega$ (i.e. $u = \1_\Omega$)
corresponding to a solution of the NP-hard problem \eqref{NPCheeger}.

The subdifferential of the energy functional $\vert \cdot \vert_{TV}$ is minus the $1$-Laplacian in graphs. Using the nonlinear eigenvalue problem  $\lambda \, {\rm sign}(u) \in -\Delta_1 u$, the theory of $1$-Spectral Clustering is developed in \cite{Chang1},  \cite{Changetal01},  \cite{ChSZ} and \cite{HB}. For a generalization of the above results to the framework of random walk spaces see \cite{MST0} and \cite{MST1}.

The only results aboutthe Cheeger cut problem in metric graphs that we know are the  ones given by Del Pezzo and  Rossi \cite{DelPR1} in which they study the first nonzero eigenvalue of the $p$-Laplacian on a quantum graph with Kirchoff boundary conditions on the vertices and study the Cheeger cut problem, taking the limit as $p \to 1$ of the eigenfunctions. Now, as we will see later, their  concept of total variation of a function of bounded variation in a metric graph is not  clear and, consequently, also their concept of perimeter (see Remark \ref{DPRossi}).  Here we use a different concept of total variation for functions in metric graph, proposed in \cite{Mazon}, and consequently of perimeter, that  takes into account the jumps of the function at the vertices.

Following the work by Nicaise \cite{Nicaise}, where  a Cheeger inequality in metric graphs is obtained, there have been very few results in this direction (see \cite{KM}, \cite{KN}  and \cite{Post}).

On the other hand, the Cheeger paper \cite{Cheeger} also motivated the so-called {\it Cheeger problem}. Given a bounded domain $\Omega \subset \R^N$, the {\it Cheeger constant} of $\Omega$ is defined as
$$h_1(\Omega) := \inf \left\{ \frac{\mbox{Per}(E)}{\vert E\vert} \ : \ E \subset \Omega , \ E \ \hbox{ with finite perimeter,} \ \vert E \vert > 0 \right\},$$
where $\mbox{Per}(E)$ is the perimeter of $E$ and $\vert E\vert$ its Lebesgue measure.
Any set $E \subset \Omega$  such that
$$\frac{\mbox{Per}(E )}{\vert E  \vert} = h_1(\Omega),$$  is called a {\it Cheeger set} of $\Omega$. Furthermore, we say that $\Omega$ is {\it calibrable}
if it is  a Cheeger set of itself, that is, if
$$\frac{\mbox{Per}(\Omega )}{\vert \Omega  \vert} = h_1(\Omega).$$ We shall generically refer to the {\it Cheeger problem}, as far as the computation or estimation of $h_1(\Omega)$, or the characterization of Cheeger sets of $\Omega$, are concerned. In the last year there has been a lot of literature on the Cheeger problem, see \cite{Leonardi} and \cite{Parini} for surveys about the Cheeger problem and \cite{MRT}, \cite{MRTLibro} for the nonlocal Cheeger problem.

It is well known that the Cheeger constant of $\Omega$  is the limit of the sequence of first
eigenvalues of the $p$-Laplacian (with Dirichlet conditions) when $p$ tends to $1$, see \cite{KF}. A similar result has been obtained by  Del Pezzo and  Rossi in \cite{DelPR2}, in the context of metric graphs, but here again the problem is their concept of perimeter in metric graphs. For the Cheeger problem in random walk spaces, that has as a particular case the weighted graphs, see \cite{MST1}.

The aim of this paper is to study the Cheeger cut and Cheeger problem in metric graphs. We introduce the concepts of Cheeger and calibrable sets in metric graphs and  we also study the eigenvalue problem  whereby we give a method to solve the  optimal Cheeger cut problem. To do that we work in the framework we developed in \cite{Mazon} to study the total variation flow in metric graphs.

The structure of the paper is as follows. In Section \ref{prelim} we recall the notion of metric graphs and the results about functions of bounded variation in metric  graphs that we need. Then, in Section \ref{Cheeger problem} we study the Cheeger problem. We introduce the concepts of Cheeger and calibrable sets  in metric graphs, we give different characterizations of the Cheeger constant of a set and its relation with the Max-Flow  Min-Cut Theorem and,  moreover, we characterize the calibrable sets. Section \ref{Eigenpair} is devoted to the eigenvalue problem for the $1$-Laplacian in metric graphs and its relations with  the Cheeger problem.  In Section \ref{Cheegerconstant} we study the Cheeger cut in metric graphs. We obtain a characterization similar to the one obtained in \cite{SB1} for weighted graphs, which  allows us to  prove the existence of  an optimal Cheeger cut, and its relation with the eigenvalue problem for the $1$-Laplacian obtaining similar results to the  ones in \cite{Chang1} for  weighted graphs, whereby we give a method to solve the  optimal Cheeger cut problem.
 Finally, we also obtained a Cheeger Inequality in metric graphs.

\section{Preliminaries}\label{prelim}

In this section, after giving the basic concepts of metric graphs, we recall the results about total variation functions introduced in \cite{Mazon} that is the framework in which we developed our work.

\subsection{Metric  Graphs} We  recall here some basic knowledge about metric graphs,
see for instance
\cite{BK} and  the references therein.

A graph $\Gamma$ consists of a finite or countable infinite set
of vertices $\V(\Gamma)=\{\v_i\}$ and a set of edges $\E(\Gamma)=\{\e_j\}$ connecting the
vertices. A graph $\Gamma$ is said to be a finite graph if the number of edges
and the number of vertices are finite.
 An edge and a vertex on
that edge are called incident. We will denote $\v\in \e$ when the edge $\e$ and the vertex $\v$ are incident.
We define $\E_{\v}(\Gamma)$ as the set of all edges incident to $\v$, and the {\it degree} of $\v$ as $d_\v:=  \sharp \E_{\v}(\Gamma)$. We define the {\it boundary} of $V(\Gamma)$ as
$$\partial V(\Gamma):= \{ \v \in V(\Gamma) \ : \ d_\v =1 \},$$
and its {\it interior} as
$${\rm int}( V(\Gamma)) := \{ \v \in V(\Gamma) \ : \ d_\v > 1 \}.$$

We will assume the absence of loops,
since if these are present, one can break them into pieces by introducing
new intermediate vertices. We also assume the absence of multiple edges.

A {\it walk} is a sequence of edges $\{\e_1,\e_2,\e_3,\dots\}$ in which, for each $i$ (except the last), the end of  $\e_i$  is
the beginning of $\e_{i+1}$. A {\it trail} is a walk in which no edge is repeated.
A {\it path} is a trail in which no vertex is repeated.

From now on we will deal with a connected, compact and metric graph~$\Gamma$:

\noindent $\bullet$ A  graph $\Gamma$ is a metric graph if
\begin{enumerate}
	\item each edge $\e$ is assigned  with a positive length $\ell_{\e}\in(0,+\infty];$

\item   for each edge $\e$, a coordinate is assigned to each point of it, including its vertices. For that purpose, each   edge $\e$ is identified with an ordered pair
$(\vi_{\e},\vf_{\e})$ of vertices, being $\vi_{\e}$ and $\vf_{\e}$ the
initial and terminal
vertex of $\e$ respectively, which has no  sense of meaning when travelling along the path  but allows us to define  coordinates by means of an increasing function
	$$
	\begin{array}{rlcc}
 c_\e:&\e&\to& [0,\ell_\e]\\
 &x&\rightsquigarrow& x_{\e}
 \end{array}
$$
such that, letting $c_\e(\vi_\e):=0$ and $c_\e(\vf_\e):=\ell_{\e}$, it is exhaustive;  $x_{\e}$ is called the coordinate of the point $x\in \e$.
\end{enumerate}

\noindent $\bullet$ A graph is said to be
connected if a path exists between every pair of vertices, that is, a graph which is
connected in the usual topological sense.

\noindent  $\bullet$ A compact metric graph is a finite metric graph whose edges  all have finite
length.

\medskip

If a sequence of edges $\{\e_j\}_{j=1}^n$  forms a path, its length is
defined as $\sum_{j=1}^n\ell_{\e_j}.$  The length of a metric graph, denoted $\ell(\Gamma)$, is the sum of the length of all its edges. Sometime we identify $\Gamma$ with $$\Gamma \equiv\bigcup_{\e \in E(\Gamma)} \e.$$

Given a set $A \subset \Gamma$, we define its {\it length} as
$$\ell (A):= \sum_{ \e \in E(\Gamma), A \cap \e \not=\emptyset} \mathcal{L}^1( c_{\e}(A \cap \e)).$$

For two vertices $\v$ and $\hat\v,$ the distance between $\v$ and $\hat \v$, $d_\Gamma(\v,\hat \v)$, is defined as the
minimal length of the paths connecting them.  Let us be more precise and consider $x$, $y$ two points in the graph $\Gamma$.

-if $x,y\in\e$ (they belong to the same edge, note that they can be vertices), we define {\it the distance-in-the-path-$\e$}   between $x$ and $y$  as $$\hbox{dist}_{\e}(x,y):=|y_\e-x_\e|;$$

-if $x\in \e_a$, $y\in \e_b$, with $\e_a$ and $\e_b$ different edges, let $P=\{\e_a,\e_1,\dots,\e_{n},\e_b\}$ be a path ($n\ge 0$) connecting them.
 Let us call $\e_{0} = \e_a$ and $\e_{n+1}= \e_b$. Following the definition given above for a path, set $\v_{0}$ the vertex that is the end of  $\e_0$  and
the beginning of  $\e_{1}$ (note that these vertices need not be the  terminal and the initial vertices of the edges that are taken into account), and   $\v_{n}$   the vertex that is the end of  $\e_n$ and the beginning of $\e_{n+1}$.
We will say that   the {\it distance-in-the-path-$P$} between $x$ and $y$ is equal to
$$\hbox{dist}_{\e_0}(x,\v_0)+ \sum_{1\le j\le n}\ell_{\e_j}+ \hbox{dist}_{\e_{n+1}}(\v_n,y) .$$
We define the distance between $x$ and $y$, that we will denote by $d_\Gamma(x,y)$, as the infimum  of all the {\it distances-in-paths} between $x$ and $y$, that is,
$$
\begin{array}{lr}d_\Gamma(x,y)&= \inf\Big\{ \hbox{dist}_{\e_0}(x,\v_0)+ \sum_{1\le j\le n}\ell_{\e_j}+ \hbox{dist}_{\e_{n+1}}(\v_n,y):\qquad\qquad \\[10pt] & \qquad\qquad \hbox{ $\{\e_0,\e_1,\dots,\e_{n},\e_{n+1}\}$   path connecting $x$ and $y$} \Big\}.
\end{array}
$$

 We remark that the distance between two points $x$ and $y$ belonging to the same edge $\e$ can be strictly smaller than $|y_\e-x_\e|$. This happens when there is a path connecting them (using  more edges than $\e$) with length smaller than $|y_\e-x_\e|$.

A function $u$ on a metric graph $\Gamma$ is a collection of functions $[u]_{\e}$
defined on
$(0,\ell_{\e})$ for all $\e\in \E(\Gamma),$
not just at the vertices as in discrete models.

Throughout this work, $ \int_{\Gamma} u(x)  dx$ or  $ \int_{\Gamma} u$ denotes
$ \sum_{\e\in \E(\Gamma)} \int_{0}^{\ell_{\e}} [u]_{\e}(x_\e)\, dx_\e$. Note that given $\Omega \subset \Gamma$, we have
$$\ell(\Omega) = \int_\Gamma \1_\Omega dx.$$

Let $1\le p\le +\infty.$ We say that $u$ belongs to  $L^p(\Gamma)$ if
$[u]_{\e}$ belongs to $L^p(0,\ell_{\e})$ for all $\e\in \E(\Gamma)$ and
\[
	\|u\|_{L^{p} (\Gamma)}^p\coloneqq\sum_{\e\in \E(\Gamma)}
	\|[u]_{\e}\|_{L^{p}(0,\ell_{\e})}^p<+\infty.
\]
 The Sobolev space $W^{1,p}(\Gamma)$ is defined as
the space of  functions $u$ on $\Gamma$ such  that
$[u]_{\e}\in W^{1,p}(0,\ell_{\e})$
for all $\e\in \E(\Gamma)$ and
\[
	\|u\|_{W^{1,p}(\Gamma)}^p\coloneqq \sum_{\e\in \E(\Gamma)}
	\|[u]_{\e}\|_{L ^p(0,\ell_{\e})}^p+\|[u]_{\e}{}^\prime\|_{L ^p(0,\ell_{\e})}^p<+\infty.
\]
The space $W^{1,p}(\Gamma)$ is a Banach space for $1 \le p \le\infty$.
It is reflexive for $1 < p < \infty$ and separable for $1 \le p < \infty.$
Observe that in the definition of $W^{1,p}(\Gamma)$ we does not assume the continuity at the vertices.

A quantum graph is a metric graph $\Gamma$ equipped with a
differential operator
acting on the edges together with vertex conditions.
In this work, we will consider the $1-$Laplacian differential operator given by
\[
	\Delta_1 u(x):=
	\left(\frac{ u^{\prime}(x)}{|u^{\prime}(x)|}\right)^{\prime},
\]
on each edge.

\subsection{BV functions and integration by parts}

We need to recall the concept of bounded variation  functions and  their total variation in metric  graphs that we introduce in \cite{Mazon} since this is the framework to study the Cheeger problem.

For bounded variation functions of one variable we follow \cite{AFP}.
Let $I \subset \R$  be an interval, we say that a
function $u \in L^1(I)$ is of bounded variation if its
distributional derivative $Du$ is a Radon measure on $I$ with
bounded total variation $\vert Du \vert (I) < + \infty$.
 We denote
by $BV(I)$ the space of all  functions of bounded variation in $I$.
It is well known (see \cite{AFP}) that given $u \in BV(I)$ there
exists $\overline{u}$ in the equivalence class of $u$, called a
good representative of $u$, with the following properties. If $J_u$
is the set of atoms of $Du$, i.e., $x \in J_u$ if and only if $Du(\{
x \}) \not= 0$, then $\overline{u}$ is continuous in $I \setminus
J_u$ and has a jump discontinuity at any point of $J_u$:
$$\overline{u}(x_{-})  := \lim_{y \uparrow x}\overline{u}(y) = Du(]a,x[), \ \ \ \
\ \overline{u}(x_{+}) := \lim_{y \downarrow x}\overline{u}(y) =
Du(]a,x]) \ \ \ \forall \, x \in J_u,$$ where by simplicity we are
assuming that $I = ]a,b[$. Consequently,
$$\overline{u}(x_{+}) - \overline{u}(x_{-})
 = Du(\{ x \}) \ \ \ \forall \, x \in J_u.$$
Moreover, $\overline{u}$ is differentiable at ${\mathcal L}^1$
     a.e. point of $I$, and the derivative $\overline{u}'$ is the
density of $Du$ with respect to ${\mathcal L}^1$.
For $u \in BV(I)$, the  measure $Du$ decomposes into its absolutely continuous and
singular parts $Du = D^a u + D^s u$. Then $D^a u = \overline{u}' \ {\mathcal L}^1$. We
also split $D^su$ in two parts: the {\it jump} part $D^j u$ and the {\it Cantor} part
$D^c u$. $J_u$ denotes the set of atoms of $Du$.

 It is well known (see for instance \cite{AFP}) that
$$D^j u = Du \res J_u = \sum_{x \in J_u} \overline{u}(x_{+}) - \overline{u}(x_{-}),$$
and also,
$$\vert Du \vert (I) = \vert D^au \vert (I) + \vert D^j u \vert (I) + \vert D^c u \vert
(I) $$ $$= \int_a^b \vert \overline{u}'(x) \vert \, dx + \sum_{x
\in J_u} \vert \overline{u}(x_{+}) - \overline{u}(x_{-}) \vert +
\vert D^c u \vert (I).$$
Obviously, if $u \in BV(I)$ then $u \in W^{1,1}(I)$ if and only if
$D^su \equiv 0$, and in this case we have $Du = \overline{u}' \
{\mathcal L}^1$.

A measurable subset $E \subset I$ is a set of {\it finite perimeter} in $I$ if $\1_E \in BV(I)$, and its perimeter is defined as
$${\rm Per}(E, I):= \vert D \1_E \vert (I).$$

From now on, when we deal with point-wise valued $BV$-functions we  shall always use the
good representative.

Given $\z \in W^{1,2}(]a,b[)$ and $u \in BV(]a,b[)$, by $\z Du$ we mean the Radon measure
in $]a,b[$ defined as
$$\langle \varphi, \z Du \rangle := \int_a^b \varphi \z \, Du \ \
\ \ \ \ \forall \, \varphi \in C_c(]a,b[).$$

Note that if $\varphi \in \mathcal{D}(]a,b[)$, then
$$\langle \varphi, \z Du \rangle = - \int_a^b u \z^{\prime} \varphi dx - \int_a^b u \z \varphi^{\prime} dx,$$
which is  the definition given by Anzellotti in \cite{Anzellotti}.

Working as in \cite[Corollary 1.6]{Anzellotti}, it is easy to see that
\begin{equation}\label{boound}
\vert \z Du \vert (B) \leq \Vert \z \Vert_{L^{\infty}(]a,b[)} \vert Du \vert (B) \quad \hbox{for all Borelian} \ B \subset ]a,b[.
\end{equation}

Then, $\z Du$ is absolutely continuous  with  respect to the measure $\vert Du \vert$.

The following result was given in \cite[Proposition 2.1]{Mazon}

\begin{proposition}\label{megusta1} Let $\z_n \in W^{1,2}(]a,b[)$. If
$$\lim_{n \to \infty}\z_n = \z \quad \hbox{weakly$^*$ in} \ L^\infty (]a,b[),$$
and
$$\lim_{n \to \infty}\z^{\prime}_n = \z^{\prime} \quad \hbox{weakly in} \ L^1 (]a,b[),$$
then for every $u \in BV(]a,b[)$, we have
$$
\z_n Du \to \z Du \quad \hbox{as measures},
$$
and
$$
\lim_{n \to \infty} \int_a^b \z_n Du = \int_a^b\z Du.
$$
\end{proposition}

 We need the following
integration by parts formula, which can be proved using a suitable
regularization of $u \in BV(I)$ as in the proof of \cite[Theorem 1.9]{Anzellotti} (see also  Theorem C.9. of
\cite{ACMBook}).

\begin{lemma}\label{IntBP} If $\z \in W^{1,2}(]a,b[)$ and $u \in
BV(]a,b[)$, then
\begin{equation}\label{EIntBP}
\int_a^b \z Du + \int_a^b u(x) \z^{\prime}(x) \, dx = \z(b) u(b_{-})- \z(a) u(a_{+}).
\end{equation}
\end{lemma}

\begin{definition}{\rm
We define the set of {\it bounded variation function} in $\Gamma$ as
$$BV(\Gamma):= \{ u \in L^1(\Gamma) \ : \ [u]_{\e}\in BV(0,\ell_{\e}) \ \hbox{for all} \ \e\in \E(\Gamma) \}.$$

Given $u \in BV(\Gamma)$, for $\e \in E_\v$, we define
$$[u]_\e(\v) := \left\{ \begin{array}{ll} [u]_\e(0+), \quad &\hbox{if} \ \ \v = \vi_{\e} \\[10pt] [u]_\e(\ell_\e-), \quad &\hbox{if} \ \ \v = \vf_{\e}. \end{array}  \right.$$

For $u \in BV(\Gamma)$, we define
$$\vert D u \vert (\Gamma):=  \sum_{\e\in \E(\Gamma)} \vert D [u]_{\e}  \vert(0,\ell_{\e}).$$
We also write $$\vert D u \vert (\Gamma) =\int_{\Gamma} |Du|.$$

}
\end{definition}

Obviously, for $u \in BV(\Gamma)$, we have
$$
 \vert D u \vert (\Gamma)= 0 \ \iff \ [u]_\e \ \hbox{is constant in} \ (0, \ell_\e), \ \ \forall \, \e \in E(\Gamma).
$$

$BV(\Gamma)$ is a Banach space  with  respect to the norm

	$$\|u\|_{BV(\Gamma)}\coloneqq \Vert u \Vert_{L^1(\Gamma)} +  \vert D u \vert (\Gamma).$$

\begin{remark}{\rm
Note that we do not include a continuity condition at the vertices as in the definition of the spaces $BV(\Gamma)$.  This is due to the fact that, if we include the continuity  at the vertices, then typical functions of bounded variation such as the functions of the form $\1_D$ with $D \subset \Gamma$ such that $\v \in D$, being $\v$ a common vertex to two edges, would not be elements of $BV(\Gamma)$.
}
$\blacksquare$
\end{remark}

By the Embedding Theorem for $BV$-function (cf. \cite[Corollary 3.49, Remark 3.30]{AFP}), we have the following result.
\begin{theorem}\label{embedding} The embedding $BV(\Gamma) \hookrightarrow L^p(\Gamma)$ is continuous for $1\leq p \leq \infty$, being compact for $1 \leq p < \infty$. Moreover, we also have the following Poincar\'{e} inequality:
$$
\Vert u - \overline{u} \Vert_p \leq C \vert D u \vert (\Gamma) \quad \forall \, u \in BV(\Gamma), \quad 1 \leq p \leq \infty,
$$
where
$$ \overline{u}:= \frac{1}{\ell(\Gamma)} \int_\Gamma u(x) dx.$$

\end{theorem}

Let us point out that in metric graphs $\vert D u \vert (\Gamma)(u)$ is not the good definition of total variation of $u$ since  it does not measure the jumps of the function at the vertices. In \cite{Mazon}, in order to give a definition of total variation of a function $ u \in BV(\Gamma)$ that takes into account the jumps of the function  at the vertices, we gave a Green's formula like the one  obtained by Anzellotti in \cite{Anzellotti} for $BV$-functions in Euclidean spaces. To do that we start by defining the pairing $\z Du$ between  an element $\z \in W^{1,1}(\Gamma)$ and
a BV function $u$. This will be a metric graph analogue of the classic Anzellotti pairing
introduced in \cite{Anzellotti}.

\begin{definition}{\rm
For $\z \in W^{1,2}(\Gamma)$ and $u \in BV(\Gamma)$, we define $\z Du:= ( [\z]_\e, D[u_\e])_{\e \in E(\Gamma)} $, that is, for $\varphi \in C_c(\Gamma)$,
$$\langle \z Du, \varphi \rangle = \sum_{\e\in \E(\Gamma)} \int_0^{\ell_{\e}} \varphi_\e[\z]_\e \, D[u]_\e.$$
We have that $\z Du$ is a Radon measure in $\Gamma$
and
$$\int_\Gamma \z Du :=  \sum_{\e\in \E(\Gamma)} \int_0^{\ell_{\e}} [\z]_\e \, D[u]_\e.$$

}
\end{definition}

By \eqref{boound}, we have

$$
\left\vert \int_{\Gamma} \z Du \right\vert \leq \Vert \z \Vert_{L^{\infty}(\Gamma)} \vert D u \vert (\Gamma). $$

Then,  $\z Du$ is absolutely continuous  with  respect to the measure $\vert Du \vert$.

Given $\z \in W^{1,2}(\Gamma)$, for $ \e \in E_\v$, we define
$$[\z]_\e (\v):= \left\{ \begin{array}{ll}[\z]_\e(\ell_{\e}) \quad &\hbox{if} \ \ \v = \vf_\e, \\[10pt] -[\z]_\e(0),\quad &\hbox{if} \ \ \v = \vi_\e.  \end{array} \right..$$

By Lemma \ref{IntBP}, we have
            $$\int_{\Gamma} \z Du : =  \sum_{\e\in \E(\Gamma)} \int_0^{\ell_{\e}}[\z]_\e \, D[u]_\e $$ $$= - \sum_{\e\in \E(\Gamma)} \int_0^{\ell_{\e}}[u]_\e(x) ([\z]_\e)^{\prime}(x) dx+ \sum_{\e\in \E(\Gamma)} ( [\z]_\e(\ell_{\e}) [u]_\e((\ell_{\e})_{-}) - [\z]_\e(0) [u]_\e (0_+) )$$ $$= - \int_\Gamma u\z^{\prime} + \sum_{\v \in V(\Gamma)} \sum_{\e\in \E_\v(\Gamma)} [\z]_\e(\v) [u]_\e(\v).$$

Then,
 if we define
$$\int_{\partial \Gamma} \z u:=\sum_{\v \in V(\Gamma)} \sum_{\e\in \E_\v(\Gamma)} [\z]_\e(\v) [u]_\e(\v),$$
for $\z \in W^{1,2}(\Gamma)$ and $u \in BV(\Gamma)$,
we have  the following {\it Green's formula}:
\begin{equation}\label{intbpart}
\int_{\Gamma} \z Du + \int_\Gamma u\z^{\prime} = \int_{\partial \Gamma} \z u.
\end{equation}

 We define
 $$X_0(\Gamma):= \{ \z \in W^{1,2}(\Gamma) \  : \ \z(\v) =0, \ \ \forall \v \in  V(\Gamma)\}.$$

 For $u \in BV(\Gamma)$ and $\z \in X_0(\Gamma)$, we have the following {\it Green's formula}
 \begin{equation}\label{0intbpart}
\int_{\Gamma} \z Du + \int_\Gamma u\z^{\prime} = 0.
\end{equation}

We consider now the elements of $W^{1,2}(\Gamma)$ that  satisfy a  {\it Kirchhoff condition}, that is, the set
 $$X_K(\Gamma):= \left\{ \z \in W^{1,2}(\Gamma) \  : \ \sum_{\e \in  E_\v(\Gamma)} [\z]_\e(\v) =0, \ \ \forall \v \in V(\Gamma) \right\}.$$

 Note that if $\z \in X_K(\Gamma)$, then $\z(\v) =0$ for all $\v \in \partial V(\Gamma)$. Therefore,
for $u \in BV(\Gamma)$ and $\z \in X_K(\Gamma)$, we have the following {\it Green's formula}
\begin{equation}\label{Kintbpart}
\int_{\Gamma} \z Du + \int_\Gamma u\z^{\prime} = \sum_{\v \in {\rm int}(V(\Gamma))}   \sum_{\e\in \E_\v(\Gamma)} [\z]_\e(\v) [u]_\e(\v).
 \end{equation}

 Now, for $\v \in {\rm int}(V(\Gamma))$, we have
 $$\sum_{\e \in  E_\v(\Gamma)} [\z]_\e(\v) [u]_{\hat{\e}}(\v) =0, \quad \hbox{for all} \ \hat{\e} \in  E_\v(\Gamma).$$
 Hence
 $$\sum_{\e\in \E_\v(\Gamma)} [\z]_\e(\v) [u]_\e(\v) = \frac{1}{d_\v} \sum_{\hat{\e}\in \E_\v(\Gamma)} \sum_{\e \in  E_\v(\Gamma)}[\z]_\e(\v)  \left( [u]_\e(\v) - [u]_{\hat{\e}}(\v) \right).$$
 Therefore, we can rewrite Green's formula \eqref{Kintbpart} as

$$
\int_{\Gamma} \z Du + \int_\Gamma u\z^{\prime} = \sum_{\v \in {\rm int}(V(\Gamma))} \frac{1}{d_\v} \sum_{\hat{\e}\in \E_\v(\Gamma)} \sum_{\e \in  E_\v(\Gamma)}[\z]_\e(\v)  \left( [u]_\e(\v) - [u]_{\hat{\e}}(\v) \right).
$$

\begin{remark}{\rm
Given a function $u$ in the metric graph $\Gamma$, we say that $u$ is {\it continuous at the vertex} $\v$, if
$$[u]_{\e_1}(\v)  = [u]_{\e_2}(\v), \quad \hbox{for all} \ \e_1, \e_2 \in E_\v(\Gamma).$$
We denote this common value as $u(\v)$. We denote by $C({\rm int}(V(\Gamma)))$ the set of all functions in $\Gamma$ continuous at the vertices $\v \in {\rm int}(V(\Gamma))$

Note that if $u \in BV(\Gamma) \cap C({\rm int}(V(\Gamma)))$ and $\z \in X_K(\Gamma)$, then by \eqref{Kintbpart}, we have
$$
\int_{\Gamma} \z Du + \int_\Gamma u\z^{\prime} =0.
$$
}$\blacksquare$
\end{remark}

\begin{definition}{\rm For $u \in BV(\Gamma)$, we define its {\it total variation} as
$$
TV_\Gamma(u) =   \sup \left\{ \displaystyle \left\vert \int_{\Gamma} u(x) \z^{\prime}(x) dx \right\vert \ : \ \z \in X_K(\Gamma), \ \Vert \z \Vert_{L^\infty(\Gamma)} \leq 1 \right\}.
$$

We say that a measurable set $E \subset \Gamma$ is a {\it set of  finite perimeter} if $\1_E \in BV(\Gamma)$, and we define its {\it $\Gamma$-perimeter} as
$${\rm Per}_\Gamma (E):= TV_\Gamma (\1_E),$$
that is
\begin{equation}\label{forper}{\rm Per}_\Gamma (E) = \sup \left\{ \displaystyle \left\vert\int_{E}  \z^{\prime}(x) dx  \right\vert \ : \ \z \in X_K(\Gamma), \ \Vert \z \Vert_{L^\infty(\Gamma)} \leq 1 \right\}.
\end{equation}
}
\end{definition}

\begin{remark}\label{Complementario}{\rm We have
\begin{equation}\label{compl1}
{\rm Per}_\Gamma (E) = {\rm Per}_\Gamma (\Gamma \setminus E), \quad \hbox{for all $E \subset \Gamma$ of finite perimeter}.
\end{equation}
In fact, given $\z \in X_K(\Gamma)$ with  $\Vert \z \Vert_{L^\infty(\Gamma)} \leq 1$, by Green's formula \eqref{Kintbpart}, it is easy to see that
$$\int_\Gamma \1_E\z^{\prime} = - \int_{\Gamma} \z D \1_E + \sum_{\v \in {\rm int}(V(\Gamma))}   \sum_{\e\in \E_\v(\Gamma)} [\z]_\e(\v) [\1_E]_\e(\v)$$ $$=  - \int_{\Gamma} \z D \1_{\Gamma \setminus E} + \sum_{\v \in {\rm int}(V(\Gamma))}   \sum_{\e\in \E_\v(\Gamma)} [\z]_\e(\v) [\1_{\Gamma \setminus E}]_\e(\v) = \int_\Gamma \1_{\Gamma \setminus E}\z^{\prime}.$$
Thus, by \eqref{forper}, we have that \eqref{compl1} holds.

}$\blacksquare$
\end{remark}

\begin{remark}\label{DPRossi}{\rm In the works by  Del Pezzo and  Rossi \cite{DelPR1} and \cite{DelPR2} it is not clear what is their concept  of functions of bounded variation on $\Gamma$ and their
  total variation. They refer to the monograph \cite{AFP} for the precise definition.  However, in \cite{AFP} only the case of functions of bounded variation in  the Euclidean space is  studied. Now, reading their works it seems that for them the space of the bounded variation functions in $\Gamma$ coincides with our space $BV(\Gamma)$, but they do not make it clear if they assume continuity  at the vertices. Their total variation of $u \in BV(\Gamma)$ is $\vert Du \vert(\Gamma)$  which does not take into account the jumps at the vertices.
}$\blacksquare$
\end{remark}

As  a consequence of the above definition, we have the following result.

\begin{proposition}\label{lsc1} $TV_\Gamma$ is lower semi-continuous   with  respect to the weak convergence in $L^1(\Gamma)$.
\end{proposition}

As in the local case, we have obtained in \cite{Mazon} the following coarea formula relating the total variation of a function with the perimeter of its superlevel sets.

\begin{theorem}[\bf Coarea formula]\label{coarea1}
 For any $u \in L^1(\Gamma)$, let $E_t(u):= \{ x \in \Gamma \ : \ u(x) > t \}$. Then,
\begin{equation}\label{coaerea}
TV_\Gamma(u) = \int_{-\infty}^{+\infty} {\rm Per}_\Gamma(E_t(u))\, dt.
\end{equation}
\end{theorem}

We introduce now

$$JV_\Gamma (u):=  \sum_{\v \in {\rm int}(V(\Gamma))} \frac{1}{d_{\v}} \sum_{\e, \hat{\e} \in E_\v(\Gamma)} \vert [u]_{\e}(\v) - [u]_{\hat{\e}}(\v) \vert.$$

Note that $JV_\Gamma (u)$ measures, in a weighted way, the jumps of u at the vertices. The following results was proved in \cite{Mazon}.

 \begin{proposition} For $u \in BV(\Gamma)$, we have
\begin{equation}\label{Form1}
\vert Du \vert(\Gamma) \leq TV_\Gamma(u) \leq \vert Du \vert(\Gamma) + JV_\Gamma (u).
\end{equation}

 If $u \in BV(\Gamma) \cap C({\rm int}(V(\Gamma)))$, then
 \begin{equation}\label{Form1zero}
TV_\Gamma(u) =\vert Du \vert(\Gamma).
\end{equation}

If $\Gamma$ is linear, that is $d_\v =2$ for all  $\v \in{\rm int}(V(\Gamma))$, then
\begin{equation}\label{Form1Igual}
TV_\Gamma(u) = \vert Du \vert(\Gamma) + JV_\Gamma (u).
\end{equation}
 \end{proposition}

 \begin{corollary}\label{cojonudo} For $u \in BV(\Gamma)$, we have
\begin{equation}\label{conssttLinear}TV_\Gamma (u) =0 \ \iff \ u \ \hbox{is constant}.
\end{equation}
Then
\begin{equation}\label{conssttLinear1}{\rm Per}_\Gamma(E) =0 \ \iff \ E = \Gamma.\end{equation}
\end{corollary}

 In \cite{Mazon} we give an example showing that
the equality \eqref{Form1Igual} does not holds if $ u \not\in C({\rm int}(V(\Gamma)))$ or there exists $\v \in{\rm int}(V(\Gamma))$ with $d_\v \geq 3$.

\subsection{The $1$-Laplacian in metric graphs.}
In \cite{Mazon}, in order to study the total variation flow in the metric graph $\Gamma$ we have introduced the energy functional  $\mathcal{F}_\Gamma : L^2(\Gamma) \rightarrow [0, + \infty]$ defined by
$$\mathcal{F}_\Gamma(u):= \left\{ \begin{array}{ll} \displaystyle
TV_\Gamma(u)
 \quad &\hbox{if} \ u\in  BV(\Gamma), \\[10pt] + \infty \quad &\hbox{if } u\in L^2(\Gamma)\setminus BV(\Gamma), \end{array} \right.$$
which is convex and lower semi-continuous, and  we have obtained the following characterization of the subdifferential of $\mathcal{F}_\Gamma $.

    \begin{theorem}\label{chasubd} Let $u \in  BV(\Gamma)$ and $v \in L^2(\Gamma)$. The following assertions are equivalent:

\noindent (i) $v \in \partial \mathcal{F}_\Gamma (u)$;

\noindent
(ii) there   exists  $\z  \in X_K(\Gamma)$, $\Vert \z \Vert_{L^\infty(\Gamma)} \leq 1$ such that
\begin{equation}\label{eqq2}
   v = -\z^{\prime}, \quad  \hbox{that is,} \quad  [v]_\e = -[\z]_\e^{\prime} \ \ \hbox{in} \ \mathcal{D}^{\prime}(0, \ell_\e) \  \forall \e \in E(\Gamma)
\end{equation}
and
\begin{equation}\label{eqq1}
\int_{\Gamma} u(x) v(x) dx = \mathcal{F}_\Gamma (u);
\end{equation}

\noindent
(iii)  there exists $\z  \in X_K(\Gamma)$, $\Vert \z \Vert_{L^\infty(\Gamma)} \leq 1$  such that \eqref{eqq2} holds and
\begin{equation}\label{eqq3}
\mathcal{F}_\Gamma (u) = \int_{\Gamma}   \z Du -  \sum_{\v \in {\rm int}(V(\Gamma))} \frac{1}{d_\v} \sum_{\hat{\e}\in \E_\v(\Gamma)} \sum_{\e \in  E_\v(\Gamma)}[\z]_\e(\v)  \left( [u]_\e(\v) - [u]_{\hat{\e}}(\v) \right).
\end{equation}

Moreover, $D(\partial \mathcal{F}_\Gamma)$ is dense in $L^2(\Gamma)$.
\end{theorem}

In \cite{Mazon} was introduced the space
 $$G_m(\Gamma):= \{ v \in L^2(\Gamma) \ : \ \exists \z \in X_K(\Gamma), \   v = -\z' \ \hbox{a.e. in} \Gamma \},$$
 and consider in $G_m(\Gamma)$ the norm
$$\Vert v \Vert_{m,*} := \inf\{\Vert \z \Vert_\infty \ : \z \in X_K(\Gamma), \ v = -\z' \ \hbox{a.e. in} \ \Gamma \}.$$

In the continuous setting  this space was introduce in \cite{Meyer}.

Note that, for $v \in G_m(\Gamma)$, we have that there exists $\z_v\in X_K(\Gamma)$, such that $v = -\z'_v$ and $\Vert v \Vert_{m,*} = \Vert \z_v \Vert_\infty$.

From the proof of Theorem \ref{chasubd}, for $f \in G_m(\Gamma)$, we have

$$
 \Vert f \Vert_{m,*}:= \sup \left\{ \left\vert \int_\Gamma f(x) u(x) dx \right\vert   : u \in BV(\Gamma), \ TV_{\Gamma}(u) \leq 1\right\},
$$
 and, moreover,
\begin{equation}\label{saiiN}\partial \mathcal{F}_\Gamma (u) = \left\{ v \in L^2(\Gamma)  \ : \ \Vert v \Vert_{m,*} \leq 1, \  \int_{\Gamma} u(x) v(x) dx = TV_\Gamma (u)\right\}.\end{equation}

 \begin{definition}{\rm We define the {\it $1$-Laplacian operator} in the metric graph $\Gamma$ as
 $$(u, v ) \in \Delta_1^{\Gamma} \iff -v \in \partial \mathcal{F}_\Gamma(u),$$
 that is, if $u \in  L^2(\Gamma)\cap BV(\Gamma)$,  $v \in L^2(\Gamma)$ and there   exists  $\z  \in X_K(\Gamma)$, $\Vert \z \Vert_{L^\infty(\Gamma)} \leq 1$ such that
\begin{equation}\label{Neqq2}
   v = \z^{\prime}, \quad  \hbox{that is,} \quad  [v]_\e = [\z]_\e^{\prime} \quad   \ \hbox{in} \ \mathcal{D}^{\prime}(0, \ell_\e) \  \forall \e \in E(\Gamma)
\end{equation}
and
\begin{equation}\label{1eqq3}
\mathcal{F}_\Gamma (u) = \int_{\Gamma}   \z Du  - \sum_{\v \in {\rm int}(V(\Gamma))}  \frac{1}{d_{\v}} \sum_{\e, \hat{\e} \in E_\v(\Gamma)} [\z]_{\e} (\v)([u]_{\e}(\v) - [u]_{\hat{\e}}(\v).
\end{equation}
 }
 \end{definition}

\section{The Cheeger Problem: $\Gamma$-Cheeger and $\Gamma$-Calibrable Sets}\label{Cheeger problem}

 Given a  set $\Omega \subset \Gamma$ with    $0 < \ell(\Omega) < \ell(X)$ and ${\rm Per}_\Gamma(\Omega) >0$, we define its {\it $\Gamma$-Cheeger constant} of $\Omega$  by
\begin{equation}\label{cheeg1}h_1^\Gamma(\Omega) := \inf \left\{ \frac{{\rm Per}_\Gamma(E)}{\ell(E)} \  : \  E \subset \Omega, \   \,  \ell( E)>0 \right\}.\end{equation}

A set  $E  \subset \Omega$ achieving the infimum in \eqref{cheeg1} is said to be an {\it $\Gamma$-Cheeger set} of $\Omega$. Furthermore, we say that $\Omega$ is {\it $\Gamma$-calibrable} if it is an $\Gamma$-Cheeger set of itself, that is, if

$$h_1^\Gamma(\Omega) = \frac{{\rm Per}_\Gamma(\Omega)}{\ell(\Omega)}.$$

For ease of notation, we will denote
$$\lambda^\Gamma_\Omega:= \frac{{\rm Per}_\Gamma(\Omega)}{\ell(\Omega )},$$
for any  set $\Omega \subset \Gamma$ with $0<\ell(\Omega)$.

Note that $\Omega$ is $\Gamma$-calibrable if and only if $\Omega$  minimizes of the functional
$${\rm Per}_\Gamma (E) - \lambda^\Gamma_\Omega \ell(E)$$
on the set $E \subset \Omega$, with $\ell(E) >0$.

It is well known (see for instance \cite{ACCh}) that in $\R^N$, any Euclidean ball is a calibrable set. Let us see in the next example that this is not true, in general, in metric graphs.

\begin{example}\label{goodexam}{\rm Consider the metric graph $\Gamma$ with fourth vertices and three edges, that is $V(\Gamma) = \{\v_1, \v_2, \v_3, \v_4 \}$ and $E(\Gamma) = \{ \e_1:=[\v_1, \v_2], \e_2:=[\v_2, \v_3], \e_3:=[\v_3, \v_4] \}$, with $\ell_{\e_1} =2$, $\ell_{\e_i} =1$, $i=2,3$.

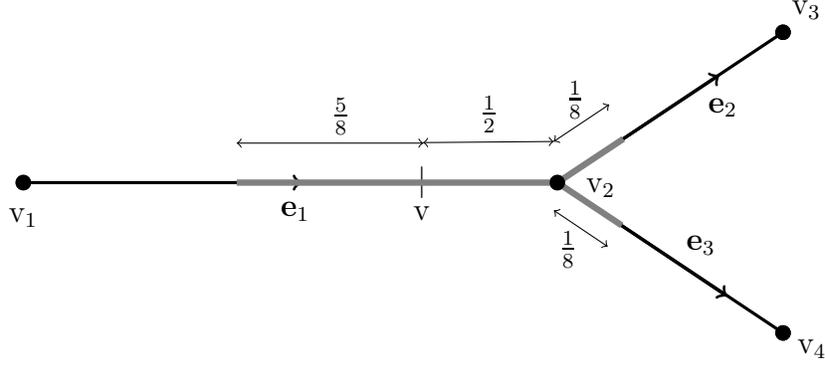
\begin{center}
\begin{tikzpicture}
 \tikzstyle{gordito2} = [line width=3]

\node (v1) at (-5.1,-1.2) {};
\node (v5) at (2.5773,-0.7977) {};

\node[below] at (v1) {$\v_1$};
\node[above] at (5.3093,1.0278) {$\v_3$};
\node[below] at (v5) {$\v_2$};
\node[above] at (5.3843,-3.4747) {$\v_4$};

\draw[gordito2]  (-5.1,-1) node (v2) {} circle (0.05);
\draw[gordito2]  (5,1) node (v3) {} circle (0.05);
\draw[gordito2]  (2,-1) node (v4) {} circle (0.05);
\draw[gordito2]  (5,-3) node (v4) {} circle (0.05);

\draw[->,line width=1.2]  (-5.1,-1)--(-1.4,-1);
\draw[line width=1.2]  (-2,-1) --(2,-1);
\draw[line width=1.2] (2,-1)--(5,-3);
\draw[->,line width=1.2] (2,-1)--(4.253,-2.5103);

\draw[line width=1.2](2,-1)--(5,1);

\draw[->,line width=1.2](2,-1)--(4.1524,0.4299);

\node at (0.2,-1) {$|$};
\node at (0.2,-1.4) {$\v$};

\draw[<->] (0.2145,-0.4674) --  (1.9512,-0.4531);
\node at (1.09,-0.0867) {$\frac12$};

\draw[<->] (2.6886,0.0386) --  (1.9512,-0.4531);
\node at (2.2324,0.1081) {$\frac18$};

\draw[<->] (2.6677,-1.8501) --  (1.9538,-1.37);
\node at (2.141,-1.8852) {$\frac18$};

\draw[<->] (-2.2636,-0.4745) --  (0.2177,-0.4745);
\node at (-0.878,-0.1128) {$\frac58$};

\node at (-1.4855,-1.3794) {$\e_1$};
\node at (4.2,0) {$\e_2$};
\node at (3.9068,-1.8274) {$\e_3$};
\node at (1.2,-4.4) {};

\draw[line width=2.4,gray]   (-2.2636,-1) --(2,-1);
\draw[line width=2.4,gray]   (2.8558,-1.5693)  --(2,-1);
\draw[line width=2.4,gray]   (2.8758,-0.4187)   --(2,-1);
\draw[gordito2]  (-5.1,-1) node (v2) {} circle (0.05);
\draw[gordito2]  (5,1) node (v3) {} circle (0.05);
\draw[gordito2]  (2,-1) node (v4) {} circle (0.05);
\draw[gordito2]  (5,-3) node (v4) {} circle (0.05);
\end{tikzpicture}
\end{center}

Consider the ball $B\left(\v,\frac58 \right)$, being $\v = c^{-1}_{\e_1}(\frac32)$. Then,
$$\lambda^\Gamma_{B\left(\v,\frac58 \right)}:= \frac{{\rm Per}_\Gamma(B\left(\v,\frac58 \right))}{\ell\left(B\left(\v,\frac58 \right) \right)} = \frac{3}{\frac58+ \frac12 +\frac28} = \frac{24}{11}.$$

Now, by \eqref{Kintbpart}, we have
 $${\rm Per}_\Gamma \left(B\left(\v,\frac12 \right)\right) = TV_\Gamma\left(\1_{B\left(\v,\frac12 \right)} \right) = \sup \left\{ \left\vert \int_\Gamma u\z^{\prime} \right\vert \ : \ \z \in X_K(\Gamma), \ \Vert \z \Vert_\infty \leq  1  \right\}$$ $$=  \sup \left\{ \left\vert - \int_{\Gamma} \z D\1_{B\left(\v,\frac12 \right)} + \sum_{\e\in \E_{\v_2}(\Gamma)} [\z]_\e(\v_2) [u]_\e(\v_2) \right\vert  \ : \ \z \in X_K(\Gamma), \ \Vert \z \Vert_\infty \leq  1  \right\}   $$ $$=  \sup \left\{ \left\vert  [\z]_{\e_1}( c^{-1}_{\e_1}(1)) + [\z]_{\e_1}(\v_2) \right\vert  \ : \ \z \in X_K(\Gamma), \ \Vert \z \Vert_\infty \leq  1  \right\} =2  $$

$$\lambda^\Gamma_{B\left(\v,\frac12 \right)}:= \frac{{\rm Per}_\Gamma(B\left(\v,\frac12 \right))}{\ell \left(B\left(\v,\frac12 \right) \right)} = 2.$$
Therefore, the ball $B\left(\v,\frac58 \right)$ is not calibrable.

It is easy to see that if $E \subset \Omega:= B\left(\v,\frac58 \right)$, with $\ell(E) >0$, then ${\rm Per}_\Gamma(E) \geq 2$, being ${\rm Per}_\Gamma(E) = 2$ if $E \subset \e_i$. Now, the subset $E \subset \Omega$ with greater volume is $E:= [c^{-1}_{\e_1}(\frac78), \v_2]$. Therefore,
$$h_1^\Gamma(\Omega) =\frac{{\rm Per}_\Gamma(E)}{\ell(E)} = \frac{2}{\frac98} = \frac{16}{9} <2.$$
Then, we have that $E$ is the $\Gamma$-Cheeger set of $\Omega$.
}$\blacksquare$
\end{example}

\begin{theorem}\label{existCS} Let $\Omega \subset \Gamma$ with ${\rm Per}_\Gamma(\Omega) >0$ and   $\ell(\Omega)>0$. There exists a Cheeger set of $\Omega$.
\end{theorem}
\begin{proof} Let $E_n \subset \Omega$ with $\ell( E_n)>0$, such that

$$h_1^\Gamma(\Omega) = \lim_{n \to \infty} \frac{{\rm Per}_\Gamma(E_n)}{\ell(E_n)}.$$

By the Embedding Theorem (Theorem \ref{embedding}), taking a subsequence if necessary,  we have that there exists $E \subset \Gamma)$, such that
 $$\1_E = \lim_{n \to \infty} \1_{E_n} \quad \hbox{in } \ L^1(\Gamma) \ \hbox{and a.e}.$$
 Then, by the lower semi-continuity of the total variation (Corollary \ref{lsc1}), we have
 $${\rm Per}_\Gamma(E) \leq \liminf_{n \to \infty} {\rm Per}_\Gamma (E_n).$$
 Therefore
 $$h_1^\Gamma(\Omega) = \frac{{\rm Per}_\Gamma(E)}{\ell(E)}.$$
 \end{proof}

\begin{remark} {\rm Let $\Omega \subset \Gamma$ with ${\rm Per}_\Gamma(\Omega) >0$ and   $\ell(\Omega)>0$. Then, if there exist $\lambda >0$ and a function $\xi : \Gamma \rightarrow \R$ such that $\xi(x) = 1$ for all $x \in \Omega$, satisfying
$$- \lambda \xi \in \Delta^\Gamma_1 \1_\Omega, \quad \hbox{in} \ \Gamma,$$
then
$$\lambda = \lambda^\Gamma_\Omega.$$
In fact, we have that  there   exists  $\z  \in X(\Gamma)$, $\Vert \z \Vert_{L^\infty(\Gamma)} \leq 1$ such that
$$
   -\lambda \xi  = \z^{\prime}, \quad \mathcal{F}_\Gamma (\1_\Omega) = \int_{\Gamma}   \z D\1_\Omega -  \sum_{\v \in {\rm int}(V(\Gamma))}  \frac{1}{d_{\v}} \sum_{\e, \hat{\e} \in E_\v(\Gamma)} [\z]_{\e} (\v)([\1_\Omega]_{\e}(\v) - [\1_\Omega]_{\hat{\e}}(\v).$$
 Then, applying  Green's formula \eqref{0intbpart}, we have
 $$\lambda \ell(\Omega) = \int_\Gamma \1_\Omega \lambda \xi dx = -\int_\Gamma \1_\Omega  \z' dx = \int_\Gamma \z D\1_\Omega - \sum_{\v \in {\rm int}(V(\Gamma))}   \sum_{\e\in \E_\v(\Gamma)} [\z]_\e(\v) [\1_\Omega]_\e(\v)$$ $$= \mathcal{F}_\Gamma (\1_\Omega) = {\rm Per}_\Gamma(\Omega).$$
}
\end{remark}

It is well known (see \cite{FK}) that the classical Cheeger constant
$$
h_1(\Omega):= \inf \left\{ \frac{Per(E)}{\vert E \vert} \, : \, E\subset \Omega, \  \vert E \vert >0 \right\},
$$
for a bounded smooth domain $\Omega \subset \R^N$, is an optimal Poincar\'{e} constant, namely, it coincides with the first eigenvalue of the $1$-Laplacian:
$$h_1(\Omega)=\Lambda_1(\Omega):= \inf \left\{ \frac{\displaystyle\int_\Omega \vert Du \vert +\displaystyle\int_{\partial \Omega} \vert u \vert d \mathcal{H}^{N-1}}{ \displaystyle\Vert u \Vert_{L^1(\Omega)}} \, : \, u \in BV(\Omega), \ \Vert u \Vert_\infty = 1 \right\}.$$

In order to get, in our context,  a  version of this result, we introduce the following constant.
For $\Omega \subset \Gamma$ with $0<\ell(\Omega)< \ell(\Gamma)$, we define
\begin{equation}\label{vvarational}\begin{array}{l}\displaystyle
\Lambda_1^\Gamma(\Omega)= \inf \left\{ TV_\Gamma(u) \ : \ u \in  BV(\Gamma), \ u= 0 \ \hbox{in} \ \Gamma \setminus \Omega, \ u \geq 0, \  \int_\Gamma u(x)  d(x) = 1, TV_\Gamma(u) >0 \right\}
\\ \\
\displaystyle \phantom{\Lambda_1^\Gamma(\Omega)}
= \inf \left\{\frac{ TV_\Gamma (u)}{\displaystyle \int_\Gamma  u(x)  dx} \ : \ u \in BV(\Gamma), \ u= 0 \ \hbox{in} \ \Gamma \setminus \Omega,\ u \geq 0, \ u\not\equiv 0, TV_\Gamma(u) >0  \right\}.
\end{array}
\end{equation}

\begin{theorem}\label{igualconsts}  Let $\Omega \subset X$ with $0 < \ell(\Omega) < \ell(X)$. Then,
\begin{equation}\label{first} h_1^\Gamma(\Omega) = \Lambda_1^\Gamma(\Omega).\end{equation}
\end{theorem}
\begin{proof}
Given a  subset $E \subset \Omega$  with $\ell(E )> 0$, we have
$$\frac{ TV_\Gamma(\1_E)}{\Vert \1_E \Vert_{L^1(X, \nu)}} = \frac{{\rm Per}_\Gamma(E)}{\ell(E)}.$$
Therefore,
\begin{equation}\label{ineqqul1}
\Lambda_1^\Gamma(\Omega) \leq h_1^\Gamma(\Omega).
\end{equation}

Suppose the another inequality does not holds. Then, there exists $u \in BV(\Gamma)$, $u= 0 \ \hbox{in} \ \Gamma \setminus \Omega$, $u \geq 0$, $u\not\equiv 0$, $TV_\Gamma(u) >0$, such that
$$\frac{ TV_\Gamma (u)}{\displaystyle \int_\Gamma  u(x)  dx} < h_1^\Gamma(\Omega).$$
Then, by the coarea formula \eqref{coaerea} and the Cavalieri's  Principle, we obtain
$$0 > TV_\Gamma (u) - h_1^\Gamma(\Omega) \int_\Gamma  u(x)  dx  =\int_0^\infty \left({\rm Per}_\Gamma(E_t(u)) - h_1^\Gamma(\Omega) \ell(E_t(u)) \right) dt \geq 0,$$
which is a contradiction, and consequently $\Lambda_1^\Gamma(\Omega) = h_1^\Gamma(\Omega)$.
\end{proof}

Let us point out that a the equality \eqref{first} was obtained in \cite[Theorem 6.2]{DelPR2}, but using a different concept of total variation and therefore of perimeter (see Remark \ref{DPRossi}).

\begin{remark}{\rm we are going to give a characterization of the solutions of the Euler-Lagrange equation of the variational problem \eqref{vvarational}. We denote by $$K_\Omega:= \left\{ u \in  BV(\Gamma), \ u= 0 \ \hbox{in} \ \Gamma \setminus \Omega, \ u \geq 0, \  \int_\Gamma u(x)  d(x) = 1, \ TV_\Gamma(u) >0 \right\},$$
and $I_{K_\Omega}$ is the indicator function of $K_\Omega$, defined by
$$I_{K_\Omega}(u) := \left\{ \begin{array}{ll} 0, \quad &\hbox{if} \ \ u \in K_\Omega, \\[10pt]   \infty, \quad &\hbox{if } \ \ u \not\in K_\Omega.\end{array} \right. $$ Then,
$$\inf \left\{ TV_\Gamma(u) \ : \ u \in  BV(\Gamma), \ u= 0 \ \hbox{in} \ \Gamma \setminus \Omega, \ u \geq 0, \  \int_\Gamma u(x)  d(x) = 1, \ TV_\Gamma(u) >0 \right\} $$ $$=  \inf \left\{ \mathcal{F}_\Gamma(u)+ I_{K_\Omega}(u) \ : \ u \in  L^2(\Gamma)\right\}.$$
Therefore,  $u$ is a minimizer of \eqref{vvarational} if and only if $0 \in \partial (\mathcal{F}_\Gamma + I_{K_\Omega})(u) = \partial \mathcal{F}_\Gamma(u) + \partial I_{K_\Omega}(u),$ where the last equality is consequence of \cite[Corollary 2.11]{Brezis}. Then,  $u$ is a minimizer of \eqref{vvarational} if and only if, $u \in K_\Omega$ and there exists $v \in \partial \mathcal{F}_\Gamma(u)$ such that $- v \in \partial I_{K_\Omega}(u)$, that is,  $\int_\Gamma uv dx \leq \int_\Gamma wv dx$ for all $w \in K_\Omega$. Now by Theorem \ref{chasubd}, we have $v \in \partial \mathcal{F}_\Gamma (u)$ if and only if there   exists  $\z  \in X_K(\Gamma)$, $\Vert \z \Vert_{L^\infty(\Gamma)} \leq 1$ such that
$$
   v = -\z^{\prime} \quad and \int_{\Gamma} u(x) v(x) dx = \mathcal{F}_\Gamma (u) $$ $$= \int_{\Gamma}   \z Du -  \sum_{\v \in {\rm int}(V(\Gamma))} \frac{1}{d_\v} \sum_{\hat{\e}\in \E_\v(\Gamma)} \sum_{\e \in  E_\v(\Gamma)}[\z]_\e(\v)  \left( [u]_\e(\v) - [u]_{\hat{\e}}(\v) \right).
$$
Consequently, we have that $u$ is a minimizer of \eqref{vvarational} if and only if $u \in K_\Omega$ and there   exists  $\z  \in X_K(\Gamma)$, $\Vert \z \Vert_{L^\infty(\Gamma)} \leq 1$ such that $$TV_\Gamma (u)   \leq \int_\Gamma \z Dw -  \sum_{\v \in {\rm int}(V(\Gamma))} \frac{1}{d_\v} \sum_{\hat{\e}\in \E_\v(\Gamma)} \sum_{\e \in  E_\v(\Gamma)}[\z]_\e(\v)  \left( [w]_\e(\v) - [w]_{\hat{\e}}(\v) \right),\quad \forall w \in K_\Omega.$$
}$\blacksquare$
\end{remark}

 The Max-Flow Min-Cut Theorem on networks due to Ford and Fulkenerson \cite{FF1}, in the continuous case was first  studied by Strang \cite{Strang1} in the particular case of the plane. Given a bounded, planar domain $\Omega$, and
given two functions $F,c : \Omega \rightarrow  \R$, we want to find the maximal value of $\lambda \in \R$ such
that there exists a vector field $V :\Omega \rightarrow  \R^2$ satisfying
$$\left\{\begin{array}{ll} {\rm div} \,  V = \lambda F \\[10pt] \Vert V \Vert_\infty \leq c. \end{array} \right.$$
The problem can be interpreted as follows: given a source or sink term $F$, we want
to find the maximal flow in $\Omega$ under the capacity constraint given by $c$. It turns out
that if $F = 1$ and $c = 1$, then the maximal value of $\lambda$ is equal to the Cheeger constant
of $\Omega$, while the boundary of a Cheeger set is the associated minimal cut (see \cite{Grieser} or \cite{Strang2}). Let us see  now that a similar result also holds in metric graphs.

\begin{theorem}\label{igualconsts2}  Let $\Omega \subset X$ with $0 < \ell(\Omega) < \ell(X)$. Then,
\begin{equation}\label{MinMax1} \begin{array}{lll} h_1^\Gamma(\Omega) &=  \sup \{ h \in \R^+ \ : \ \exists \z \in X_K(\Gamma), \ \Vert \z \Vert_\infty \leq 1, \ \z' \geq h \ \hbox{in} \ \Omega \} \\[10pt] &= \sup \left\{\frac{1}{\Vert \z \Vert_\infty }  \ : \ \z \in X_K(\Gamma), \ \z' = \1_\Omega  \right\} \\[10pt] &= \sup \left\{\frac{1}{\Vert \z \Vert_\infty }  \ : \ \z \in X_K(\Gamma), \ \z' = 1  \ \hbox{in} \ \Omega \right\}. \end{array}
\end{equation}
\end{theorem}
\begin{proof} Let
$$B:= \{ h \in \R^+ \ : \ \exists \z \in X_K(\Gamma), \ \Vert \z \Vert_\infty \leq 1, \ \z' \geq h \ \hbox{in} \ \Omega \},$$
and
$$\alpha:= \sup B.$$
Given $h \in B$ and $E \subset \Omega$ with $\ell(E) >0$, applying \eqref{forper}, we have
$$h \ell(E) = \int_E h dx \leq \int_E \z' dx \leq {\rm Per}_\Gamma (E).$$
Hence,
$$h \leq \frac{ {\rm Per}_\Gamma (E)}{\ell(E)}.$$
Then, taking  the supremum in $h$ and  the infimum in $E$, we obtain that $\alpha \leq  h_1^\Gamma(\Omega)$.

On the other hand, by Theorem \ref{igualconsts}, it is easy to see that
$$\frac{1}{ h_1^\Gamma(\Omega)} = \sup \left\{ \frac{\int_\Gamma u(x) dx}{\Vert u' \Vert_{L^1(\Gamma)} } \ : \ u \in W^{1,1}(\Gamma), \ u= 0 \ \hbox{in} \ \Gamma \setminus \Omega,\ u \geq 0, \ u\not\equiv 0, \Vert u' \Vert_{L^1(\Gamma)}  >0  \right\}$$ $$=  \sup \left\{ \int_\Omega u(x) dx \ : \ u \in W^{1,1}(\Gamma), \  \Vert u' \Vert_{L^1(\Gamma)} \leq 1,  \ u= 0 \ \hbox{in} \ \Gamma \setminus \Omega,\ u \geq 0, \ u\not\equiv 0 \right\}$$ $$= - \inf \left\{ -\int_\Omega u(x) dx \ : \ u \in W^{1,1}(\Gamma), \  \Vert u' \Vert_{L^1(\Gamma)} \leq 1,  \ u= 0 \ \hbox{in} \ \Gamma \setminus \Omega,\ u \geq 0, \ u\not\equiv 0  \right\}.$$
Then,
$$- \frac{1}{ h_1^\Gamma(\Omega)} = \inf \left\{F(u)+ G(L(u)) \ : \ u \in L^1(\Gamma) \right\}, $$
being $L : W^{1,1}(\Gamma) \rightarrow L^1(\Gamma)$ the linear map $L(u) :=u'$, $F(u):= -\int_\Gamma u \1_\Omega dx$ and $G: L^1(\Gamma) \rightarrow [0, +\infty]$ the convex function
$$G(v):= \left\{ \begin{array}{ll} 0 \quad &\hbox{if} \ \Vert v \Vert_{L^1(\Gamma)} \leq 1, \\[10pt] + \infty &\hbox{otherwise}. \end{array} \right.$$
By the Fenchel-Rockafellar duality Theorem given in \cite[Remark 4.2]{EkelandTemam}, we have
$$\inf \left\{F(u)+ G(L(u)) \ : \ u \in L^1(\Gamma) \right\} = \sup \left\{-G^*(-\z) - F^*(L^*(\z)) \ : \ \z \in L^\infty(\Gamma) \right\} $$ $$ =- \inf \left\{F^*(L^*(\z)) + G^*(-\z) \ : \ \z \in L^\infty(\Gamma) \right\}.$$
Now, $L^*(\z) = - \z'$, $G^*(\z) = \Vert \z \Vert_{L^\infty(\Gamma)}$ and
$$F^*(w) = \sup_{u \in L^1(\Gamma)} \left\{\int w u dx +\int_\Gamma u \1_\Omega dx \ : \  u \in L^1(\Gamma) \right\}.$$
Hence
$$F^*(L^*(\z)) = \sup_{u \in L^1(\Gamma)} \left\{-\int \z' u dx +\int_\Gamma u \1_\Omega dx \ : \  u \in L^1(\Gamma) \right\}.$$
Therefore,
$$\frac{1}{ h_1^\Gamma(\Omega)}  =  \inf \left\{ \Vert \z \Vert_{L^\infty(\Gamma)} \ : \ \z \in L^\infty(\Gamma), \ \z' = \1_\Omega \right\},$$
from where it follows that
$$ h_1^\Gamma(\Omega) = \sup \left\{ \frac{1}{\Vert \z \Vert_{L^\infty(\Gamma)}} \ : \ \z \in L^\infty(\Gamma), \ \z' = \1_\Omega \right\} $$ $$\leq \sup \left\{\frac{1}{\Vert \z \Vert_\infty }  \ : \ \z \in X_K(\Gamma), \ \z' = 1  \ \hbox{in} \ \Omega \right\} \leq \alpha,$$
and we finish the proof.
\end{proof}

Let us recall that, in the local case, a set $\Omega \subset \R^N$ is called {\it calibrable} if
$$\frac{\mbox{Per}(\Omega )}{\vert \Omega  \vert}  = h(\Omega):= \inf \left\{ \frac{\mbox{Per}(E)}{\vert E\vert} \ : \ E \subset \Omega , \ E \ \hbox{ with finite perimeter,} \ \vert E \vert > 0 \right\}.$$ The following characterization  of convex calibrable sets is proved in \cite{ACCh}.
\begin{theorem}\label{BCNth}(\cite{ACCh})
Given a bounded convex set $\Omega \subset \R^N$ of class $C^{1,1}$, the following assertions are equivalent:
\item(a) $\Omega $ is calibrable.
\item(b) $\1_\Omega $ satisfies $- \Delta_1 \1_\Omega   = \frac{\mbox{Per}(\Omega )}{\vert \Omega  \vert} \1_\Omega $, where $\Delta_1 u:= {\rm div} \left( \frac{Du}{\vert Du \vert}\right)$.
\end{theorem}

\begin{remark}\label{vp}{\rm By \eqref{saiiN}, we have
$$- \lambda^\Gamma_\Omega \1_\Omega \in \Delta_1^\Gamma \1_\Omega  \iff
 \Vert \lambda^\Gamma_\Omega \1_\Omega \Vert_{m,*} \leq 1, \ \hbox{and} \int_\Gamma \lambda^\Gamma_\Omega \1_\Omega \1_\Omega = TV_\Gamma(\1_\Omega). $$
Now
$$ \int_\Gamma \lambda^\Gamma_\Omega \1_\Omega \1_\Omega = \lambda^\Gamma_\Omega \ell(\Omega) = TV_\Gamma(\1_\Omega).$$
Therefore, we have
\begin{equation}\label{oole}
\begin{array}{ll}
- \lambda^\Gamma_\Omega \1_\Omega \in \Delta_1^\Gamma \1_\Omega  \iff \Vert \lambda^\Gamma_\Omega \1_\Omega \Vert_{m,*} \leq 1  \\[10pt]\iff  \sup \left\{ \left\vert \displaystyle\int_\Omega u(x) dx \right\vert   : u \in BV(\Gamma), \ TV_{\Gamma}(u) \leq 1 \right\} \leq \frac{\ell(\Omega)}{{\rm Per}_\Gamma(\Omega)}.
\end{array}
\end{equation}
$\blacksquare$
}
\end{remark}

In order to get a similar result to Theorem \ref{BCNth} we need the following concept of convexity.

\begin{definition}{\rm We say that $\Omega \subset \Gamma$ is {\it  path-convex} if for any $E \subset \Gamma$,
$$\mbox{Per}_\Gamma(\Omega \cap E) \leq \mbox{Per}_\Gamma(E).$$
}
\end{definition}

We have the following version of Theorem~\ref{BCNth}.

 \begin{theorem}\label{falsa}    Let $\Omega \subset \Gamma$ be with $0<\ell(\Omega)<\ell(\Gamma)$.  We have:
 \item{ (i)} If  $\1_\Omega$ satisfies
\begin{equation}\label{22alares} - \lambda^\Gamma_\Omega \1_\Omega \in \Delta_1^\Gamma \1_\Omega \, \quad \hbox{in} \ \Gamma,
\end{equation}
then   $\Omega$ is $\Gamma$-calibrable.

\item{ (i)}   If $\Omega$ is  path-convex  and $\Omega$ is $\Gamma$-calibrable, then equation \eqref{22alares} holds.
\end{theorem}
\begin{proof} $(i)$: For any $E \subset \Omega$ with $\ell( E)>0$, applying \eqref{oole} with $u:= \frac{\1_E}{{\rm Per}_\Gamma(E)}$, we have
$$\int_\Omega \frac{\1_E}{{\rm Per}_\Gamma(E)} \leq \frac{\ell(\Omega)}{{\rm Per}_\Gamma(\Omega)}.$$
Then,
$$\frac{{\rm Per}_\Gamma(\Omega)}{\ell(\Omega)} \leq \frac{{\rm Per}_\Gamma(E)}{\ell(E)},$$
and consequently  $\Omega$ is $\Gamma$-calibrable.

   \noindent $(ii)$: By \eqref{oole}, we need to show that Let us prove that the function $f := \lambda^\Gamma_\Omega \1_\Omega$ satisfies $\Vert f \Vert_{m,*} \leq 1$.  Indeed, if $w \in BV(\Gamma) \cap L^2(\Gamma)$ is nonnegative, by the coarea formula, we have
   $$\int_\Gamma f(x) w(x) dx = \int_0^\infty \int_\Gamma  \lambda^\Gamma_\Omega \1_\Omega \1_{E_t(w)} dx dt  =  \int_0^\infty  \lambda^\Gamma_\Omega  \, \ell(\Omega \cap E_t(w)) dt $$ $$ \leq  \int_0^\infty {\rm Per}_\Gamma(\Omega \cap E_t(w)) dt \leq \int_0^\infty {\rm Per}_\Gamma(E_t(w)) dt = TV_\Gamma (w).$$

   Splitting any function $w \in BV(\Gamma)$ into its positive and negative part, using the
above inequality one can prove that
$$\left\vert \int_\Gamma f(x) w(x) dx \right\vert \leq TV_\Gamma (w),$$
from where it follows that $\Vert f \Vert_{m,*} \leq 1$.
\end{proof}

\begin{remark}\label{more1}{\rm  (i) Note that in equation \eqref{22alares} we can change $\1_\Omega$ for a function $\xi$ such that $\xi(x) = 1$ for every $x \in \Omega$.

\noindent (ii) Let us see that this assumption $\Omega$  path-convex  is necessary for $(ii)$. For that let us give an example of a set $\Gamma$-calibrable not path-convex that verifies \eqref{22alares} but not (ii).

Consider the metric graph $\Gamma$ with two vertices and one edges, that is $V(\Gamma) = \{\v_1, \v_2 \}$ and $E(\Gamma) = \{ \e:=[\v_1, \v_2] \}$, with $\ell_{\e} =5$. Let $\Omega:= [c_\e^{-1}(1), c_\e^{-1}(2)] \cup [c_\e^{-1}(3), c_\e^{-1}(4)]$

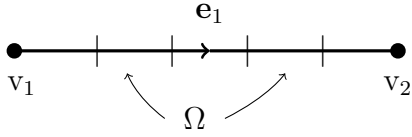
\begin{center}

\begin{tikzpicture}
 \tikzstyle{gordito2} = [line width=3]
\tikzstyle{gordito2} = [line width=3]
\node (v1) at (-5,-1.2) {};
\node (v5) at (0,-1.2) {};
\node[below] at (v1) {$\v_1$};
\node[below] at (v5) {$\v_2$};
\draw[gordito2]  (-5.1,-1) node (v2) {} circle (0.05);
\draw[gordito2]  (0,-1) node (v4) {} circle (0.05);
\draw[->,line width=1.2]  (-5.1,-1)--(-2.5,-1);
\draw[line width=1.2]  (-2.5,-1) --(0,-1);
\node at (-4,-1) {$|$};
\node at (-1,-1) {$|$};
\node at (-2.7,-1.9) {$\Omega$};
\node at (-2.5,-0.5) {$\e_1$};
\draw[->] (-3.1,-1.9) .. controls (-3.5,-1.6) and (-3.6,-1.3) .. (-3.6,-1.3);
\draw[->] (-2.3,-1.9) .. controls (-1.7,-1.6) and (-1.5,-1.3) .. (-1.5,-1.3);
\node at (-2,-1) {$|$};
\node at (-3,-1) {$|$};
\end{tikzpicture}
\end{center}

 If $E \subset \Omega$, with ${\rm Per}_\Gamma(E) >0,  \ell( E)>0$, then obviously, ${\rm Per}_\Gamma(E) \geq {\rm Per}_\Gamma(\Omega)$. Hence, $\Omega$ is $\Gamma$-calibrable. On the other hand, if $E:= [c_\e^{-1}(1), \v_2]$, we have

 $${\rm Per}_\Gamma(\Omega \cap E) = 4 > {\rm Per}_\Gamma( E) =1.$$
 Thus, $\Omega$ is not path-convex. Now by \cite[Theorem 2.11]{BF} (see also \cite{Chang0}),  $\1_\Omega$ does not satisfy
\begin{equation}\label{22alaresN} - \lambda^\Gamma_\Omega \1_\Omega \in \Delta_1^\Gamma \1_\Omega \, \quad \hbox{in} \ \Gamma,
\end{equation}
since if equation \eqref{22alaresN} has a solution, then, $\Omega$ must  be of the form $\Omega = [\a, \b]$.
}$\blacksquare$
\end{remark}

A celebrated result of De Giorgi (\cite{Giorgi}) states that, if $E$ is a set of finite perimeter
in $\R^N$ , and $E^*$  is a ball such that $|E^*| = |E|$, then ${\rm Per}(E^*) \leq {\rm Per}(E)$, with equality holding if and only if $E$ is itself a ball. This implies that
$$h(\Omega^*) \leq h(\Omega).$$

In the next example we will see that this isoperimetric inequality is not true in metric graphs.

\begin{example}\label{isop}{\rm Consider the metric graph $\Gamma$ of the Example \ref{goodexam}, that is, $V(\Gamma) = \{\v_1, \v_2, \v_3, \v_4 \}$ and $E(\Gamma) = \{ \e_1:=[\v_1, \v_2], \e_2:=[\v_2, \v_3], \e_3:=[\v_3, \v_4] \}$, with $\ell_{\e_1} =2$, $\ell_{\e_i} =1$, $i=2,3$. If $E:= [\v_1, c_{\e_1}^{-1}(\frac32)]$, we have $\ell(E) = \frac32 =  \ell B \left(\v_2, \frac12 \right)$. Now,
$${\rm Per}_\Gamma(E) =1, \quad \hbox{and} \quad {\rm Per}_\Gamma \left(B\left(\v_2, \frac12 \right)\right) =3.$$

}$\blacksquare$
\end{example}

\section{The Eigenvalue Problem for the minus $1$-Laplacian in Metric Graphs}\label{Eigenpair}

In this section we introduce the eigenvalue problem associated with the operator $-\Delta^\Gamma_1$ and its relation with the Cheeger minimization problem.

Recall that

$${\rm sign}(u)(x):=  \left\{ \begin{array}{lll} 1 \quad \quad &\hbox{if} \ \  u(x) > 0, \\ -1 \quad \quad &\hbox{if} \ \ u(x) < 0, \\ \left[-1,1\right] \quad \quad &\hbox{if} \ \ \ u = 0. \end{array}\right.$$

\begin{definition}{
A pair $(\lambda, u) \in \R \times BV(\Gamma)$ is called an {\it $\Gamma$-eigenpair} of the operator $-\Delta^\Gamma_1$ on $X$ if $\Vert u \Vert_{L^1(\Gamma)} = 1$ and    there exists $\xi \in {\rm sign}(u)$  (i.e., $\xi(x) \in {\rm sign}(u(x))$ for every $x\in \Gamma$)  such that
$$
\lambda \, \xi \in \partial \mathcal{F}_\Gamma(u) = - \Delta^\Gamma_1 u.
$$The function $u$ is called an {\it $\Gamma$-eigenfunction} and $\lambda$ an {\it $\Gamma$-eigenvalue} associated to $u$.
}
\end{definition}

   Observe that, if $(\lambda, u)$ is an $\Gamma$-eigenpair of $-\Delta^\Gamma_1$, then $(\lambda, - u)$ is also an $\Gamma$-eigenpair of $-\Delta^\Gamma_1$.

\begin{remark}\label{1257m}{\rm By Theorem \ref{chasubd}, the following statements  are equivalent:

\noindent (1) $(\lambda, u)$ is an  $\Gamma$-eigenpair of the operator  $-\Delta^\Gamma_1$ .

 \noindent (2) $u \in BV(\Gamma)$, $\Vert u \Vert_{L^1(\Gamma)} = 1$ and  there exists $\xi \in {\rm sign}(u)$  and $\z  \in X_K(\Gamma)$, $\Vert \z \Vert_{L^\infty(\Gamma)} \leq 1$ such that
\begin{equation}\label{eqq2Ei}
    \lambda \xi= -\z^{\prime},
\end{equation}
and
\begin{equation}\label{eqq1Ei}
\lambda  = TV_\Gamma (u);
\end{equation}

 \noindent (3)  $u \in BV(\Gamma)$, $\Vert u \Vert_{L^1(\Gamma)} = 1$ and  there exists  $\xi \in {\rm sign}(u)$  and  there exists $\z  \in X_K(\Gamma)$, $\Vert \z \Vert_{L^\infty(\Gamma)} \leq 1$  such that \eqref{eqq2Ei} holds and
\begin{equation}\label{eqq3Ei}
TV_\Gamma (u) = \int_{\Gamma}   \z Du  - \sum_{\v \in {\rm int}(V(\Gamma))}  \frac{1}{d_{\v}} \sum_{\e, \hat{\e} \in E_\v(\Gamma)} [\z]_{\e} (\v)([u]_{\e}(\v) - [u]_{\hat{\e}}(\v).
\end{equation}
}$\blacksquare$
\end{remark}

 \begin{proposition}\label{Peigen1}
  Let $(\lambda, u)$ be an $\Gamma$-eigenpair of  $-\Delta^\Gamma_1$. Then,
    \item(i) $\lambda = 0 \ \iff \ u \ \hbox{is constant, that is, $u= \frac{1}{\ell(\Gamma)}$, or $u=-\frac{1}{\ell(\Gamma)}$}.$

  \item(ii)   $\lambda \not= 0 \ \iff $  there exists $\xi \in {\rm sign}(u)$ such that
 $\displaystyle \int_\Gamma \xi(x) dx = 0.$
 \end{proposition}

  \begin{proof} (i) By \eqref{eqq1Ei}, if $\lambda = 0$, we have that $TV_\Gamma(u) =0$ and then, by \eqref{cojonudo}, we get that  $u$ is constant.  Thus, since $\Vert u\Vert_{L^1(\Gamma)}=1$, either $u= \frac{1}{\ell(\Gamma)}$, or $u=-\frac{1}{\ell(\Gamma)}$.  Similarly, if $u$ is constant a.e. then $TV_\Gamma(u) =0$ and, by  \eqref{eqq1Ei}, $\lambda =0$.

\noindent (ii) ($\Longleftarrow$) If $\lambda=0$, by (i), we have that $u= \frac{1}{\ell(\Gamma)}$, or $u=-\frac{1}{\ell(\Gamma)}$, and this is a contradiction with the existence of $\xi \in {\rm sign}(u)$ such that
 $  \int_\Gamma \xi(x) dx = 0$.

  \noindent($\Longrightarrow$) By Remark \ref{1257m} there exists $\xi \in {\rm sign}(u)$  and $\z  \in X_K(\Gamma)$, $\Vert \z \Vert_{L^\infty(\Gamma)} \leq 1$ satisfying \eqref{eqq2Ei}, \eqref{eqq2Ei} and \eqref{eqq3Ei}. Hence, by Green's formula \eqref{Kintbpart}, we have
  $$\lambda \int_\Gamma \xi(x) dx = - \int_\Gamma \z' =0 $$
  Therefore, since $\lambda \not= 0$,
$$\int_\Gamma \xi(x) dx =0.$$
  \end{proof}

   Recall that, given a function $u : \Gamma \rightarrow \R$,  $\mu \in \R$ is a {\it median} of $u$ with respect to a measure $\ell$ if
$$\ell(\{ x \in \Gamma \ : \ u(x) < \mu \}) \leq \frac{1}{2} \ell(\Gamma), \quad \ell(\{ x \in \Gamma \ : \ u(x) > \mu \}) \leq \frac{1}{2} \ell(\Gamma).$$
We denote by ${\rm med}_\ell (u)$ the set of all medians of $u$. It is easy to see that
$$\mu \in {\rm med}_\ell (u) \iff $$ $$ - \ell(\{ u = \mu \}) \leq \ell(\{ x \in \Gamma \ : \ u(x) > \mu \}) - \ell(\{ x \in \Gamma \ : \ u(x) < \mu \}) \leq \ell(\{ u = \mu \}),$$
from where it follows that
\begin{equation}\label{sigg1}
0 \in {\rm med}_\ell (u) \iff \exists \xi \in {\rm sign}(u) \ \hbox{such that} \ \int_\Gamma \xi(x) d x = 0,
\end{equation}

By Proposition \ref{Peigen1} and relation~\eqref{sigg1}, we have the following result that was   obtained for finite  weighted graphs by Hein and B\"uhler in \cite{HB}.

 \begin{corollary}\label{meddiaa}
 If  $(\lambda, u)$ is   an $\Gamma$-eigenpair of  $\Delta^m_1$ then
    $$ \lambda \not= 0\ \Longleftrightarrow \ 0 \in {\rm med}_\ell (u).$$
\end{corollary}

\begin{proposition} If $\left(\lambda^\Gamma_\Omega, \frac{1}{\ell(\Omega)}\1_\Omega \right)$ is a $\Gamma$-eigenpair of $-\Delta^\Gamma_1$, then $\Omega$ is $\Gamma$-calibrable.
\end{proposition}
\begin{proof} By Remark \ref{1257m} there exists $\xi \in {\rm sign}(\1_\Omega)$  and $\z  \in X(\Gamma)$, $\Vert \z \Vert_{L^\infty(\Gamma)} \leq 1$ satisfying
   $$ \lambda^\Gamma_\Omega \xi= -\z^{\prime},\quad \mathcal{F}_\Gamma \left(\frac{1}{\ell(\Omega)}\1_\Omega \right) =  \lambda^\Gamma_\Omega.$$
   Then, since $\xi =1$ in $\Omega$ and verifies
    $$ \lambda^\Gamma_\Omega \xi= -\z^{\prime},\quad \mathcal{F}_\Gamma(\1_\Omega) = {\rm Per}_\Gamma (\Omega),$$
    we have
    $$- \lambda^\Gamma_\Omega \xi \in \Delta_1^\Gamma \1_\Omega \, \quad \hbox{in} \ \Gamma.$$
    Then, by Theorem \ref{falsa} and having in mind Remark \ref{more1}, we get that $\Omega$ is $\Gamma$-calibrable.
\end{proof}

In the next example we see that, in the above proposition, the reverse implication is false in general.

\begin{example}{\rm Consider the metric graph $\Gamma$ with two vertices and one edge, that is $V(\Gamma) = \{\v_1, \v_2 \}$ and $E(\Gamma) = \{ \e:=[\v_1, \v_2] \}$, with $\ell_{\e} =6$. Let $\Omega:= [c_\e^{-1}(1), c_\e^{-1}(5)]$.

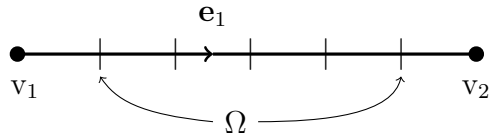
\begin{center}

\begin{tikzpicture}
  \tikzstyle{gordito2} = [line width=3]
\node (v1) at (-5,-1.2) {};
\node (v5) at (1,-1.2) {};
\node[below] at (v1) {$\v_1$};
\node[below] at (v5) {$\v_2$};
\draw[gordito2]  (-5.1,-1) node (v2) {} circle (0.05);
\draw[gordito2]  (1,-1) node (v4) {} circle (0.05);
\draw[->,line width=1.2]  (-5.1,-1)--(-2.5,-1);
\draw[line width=1.2]  (-2.5,-1) --(1,-1);
\node at (0,-1) {$|$};
\node at (-1,-1) {$|$};
\node at (-2,-1) {$|$};
\node at (-3,-1) {$|$};
\node at (-4,-1) {$|$};

\node at (-2.2,-1.9) {$\Omega$};
\node at (-2.5,-0.5) {$\e_1$};
\draw[->] (-2.5,-1.9) .. controls (-3.5,-1.8) and (-3.8,-1.6) .. (-4,-1.3);
\draw[->] (-1.9,-1.9) .. controls (-1,-1.9) and (-0.1,-1.8) .. (0,-1.3);
\node at (0,-1) {};
\node at (0,-1) {};
\node at (0,-1) {};
\node at (0,-1) {};
\node at (0,-1) {};
\node at (0,-1) {};
\node at (1,-1) {};
\node at (1,-1) {};

\end{tikzpicture}
\end{center}

Obviously, $\Omega$ is $\Gamma$-calibrable. Now, since $0 \not\in  {\rm med}_\nu (\1_\Omega)$, by Corollary \ref{meddiaa}, we have $\left(\lambda^\Gamma_\Omega, \frac{1}{\ell(\Omega)}\1_\Omega \right)$ is not a $\Gamma$-eigenpair of $\Delta^\Gamma_1$
} $\blacksquare$
\end{example}

\section{The Cheeger cut  in Metric Graphs}\label{Cheegerconstant}

We defined the {\it $\Gamma$-Cheeger constant} of $\Gamma$ as
$$
 h(\Gamma):= \inf \left\{ \frac{{\rm Per}_\Gamma(D)}{\min\{ \ell(D), \ell(\Gamma \setminus D)\}} \ : \  D \subset \Gamma, \ 0 < \ell(D) < \ell(\Gamma), \right\} $$
 or, equivalently,
 \begin{equation}\label{1523m}
 h(\Gamma)= \inf \left\{ \frac{{\rm Per}_\Gamma(D)}{\ell(D)} \ : \  D \subset \Gamma, \ 0 < \ell(D) \leq \frac{1}{2} \ell(\Gamma)\right\}.
\end{equation}

A partition $(D, \Gamma \setminus D)$ of $\Gamma$ is called a {\it  Cheeger cut} of $\Gamma$ if $D$ is a minimizer of problem \eqref{1523m}, i.e., if $0 < \ell(D) \leq \frac{1}{2} \ell(\Gamma)$ and $h(\Gamma) = \frac{{\rm Per}_\Gamma(D)}{\ell(D)}$.

Note that if  $D \subset \Gamma$, $0 < \ell(D) \leq \frac{1}{2} \ell(\Gamma)$, we have
$$\frac{{\rm Per}_\Gamma(D)}{\ell(D)} \geq \frac{1}{\frac12 \ell(\Gamma)} = \frac{2}{\ell(\Gamma)},
$$
and therefore
$$
 h(\Gamma) \geq \frac{2}{\ell(\Gamma)}.
$$

   We will now  give a variational characterization of the Cheeger constant  which for finite weighted graphs was obtained in \cite{SB1} (see also \cite{MST1}). For Riemann compact manifolds the first result of this pype was obtaine by Yau in \cite{Yau}.

\begin{theorem}\label{varchar} We have
\begin{equation}\label{minnb} h(\Gamma) =\lambda_1(\Gamma)  := \inf \left\{ TV_\Gamma(u) \ : \ u \in\Pi(\Gamma) \right\},\end{equation}
where $$\Pi(\Gamma):= \left\{ u \in BV(\Gamma)  \ :  \ \Vert u \Vert_1 = 1, \ 0 \in {\rm med}_\ell (u) \right\}.$$
Moreover, there exists a minimizer $u$ of the problem \eqref{minnb} and also  $t \geq 0$, such that $E_t(u)$ is a Cheeger cut of $\Gamma$.
\end{theorem}
\begin{proof} If $D \subset \Gamma, \ 0 < \ell(D) \leq \frac{1}{2} \ell(\Gamma)$, then $0 \in {\rm med}_\ell (\1_D)$. Thus,
$$\lambda_1(\Gamma) \leq TV_\Gamma \left(\frac{1}{\ell(D)} \1_D \right) = \frac{1}{\ell(D)} {\rm Per}_\Gamma(D)$$
and, therefore,
$$\lambda_1(\Gamma) \leq  h(\Gamma).$$

On the other hand, by the Embedding Theorem (Theorem \ref{embedding}) and the lower semi-continuity of the total variation (Corollary \ref{lsc1}), applying the Direct Method of Calculus of Variation, we have that there exists a function $u \in L^1(\Gamma)$ such that $\Vert u \Vert_1 = 1$ and $0 \in {\rm med}_\ell (u)$, such that $TV_\Gamma(u) = \lambda_1(\Gamma)$. Now, since $0 \in {\rm med}_\nu (u)$, we have $\ell(E_t(u)) \leq \frac12 \ell(\Gamma)$ for $t \geq 0$ and $\ell(\Gamma \setminus E_t(u)) \leq \frac12 \ell(\Gamma)$ for $t \leq 0$. Then by the  Coarea formula (Theorem \ref{coarea1}), the Cavalieri's  Principle and having in mind that the set  $\{ t \in \R \ : \ \ell( \{ u = t \}) > 0 \}$ is countable, we have
$$0 \leq \int_0^\infty \left({\rm Per}_\Gamma(E_t(u)) - h(\Gamma) \ell(E_t(u))\right) dt + \int_{-\infty}^{0} \left( {\rm Per}_\Gamma(X \setminus E_t(u)) - h(\Gamma) \ell(X \setminus E_t(u)) \right)\, dt $$ $$=  \int_{-\infty}^{+\infty} {\rm Per}_\Gamma(E_t(u))\, dt - h(\Gamma) \left( \int_0^\infty \ell(E_t(u)) dt + \int_{-\infty}^{0} \ell(X \setminus E_t(u))  dt \right)  $$ $$= TV_\Gamma(u) - h(\Gamma)  \left(\int_\Gamma u^+(x) dx + \int_\Gamma u^-(x) dx \right)  = TV_\Gamma(u) - h(\Gamma)\Vert u \Vert_1 = TV_\Gamma(u) - h(\Gamma) \leq 0.$$
It follows that for almost every $t \geq 0$ (in the sense of the Lebesgue measure on $\R$),
\begin{equation}\label{CCneqqul2}
{\rm Per}_\Gamma(E_t(u)) - h(\Gamma) \ell(E_t(u)) = 0.
\end{equation}
Since $u\not\equiv 0$, there must exist $t \geq 0$ such that $\ell(E_t(u)) >0$ and for which \eqref{CCneqqul2} holds.
This yields at once
$$\lambda_1(\Gamma) \leq  h(\Gamma),$$
as well as that $E_t(u)$ is a Cheeger cut of $\Gamma$.
\end{proof}
\begin{corollary} We have
\begin{equation}\label{E1mmed}
h(\Gamma) = \min \left\{ \frac{TV_\Gamma (u)}{ \Vert u - \mu \Vert_1} \ : \ u \in L^1(\Gamma), \ \mu \in  {\rm med}_\ell (u)   \right\} = \inf \left\{ \sup_{c \in \R} \frac{TV_\Gamma(u)}{\Vert u - c \Vert_1} \ : \ u \in L^1(\Gamma) \right\}.
\end{equation}
\end{corollary}
\begin{proof}  A simple calculation show that
\begin{equation}\label{E2mmed}
\lambda_1(\Gamma) = \min \left\{ \frac{TV_\Gamma (u)}{ \Vert u - \mu \Vert_1} \ : \ u \in L^1(\Gamma), \ \mu \in  {\rm med}_\ell (u)   \right\}.
\end{equation}

Let $$\alpha:=  \inf \left\{ \sup_{c \in \R} \frac{TV_\Gamma(u)}{\Vert u - c \Vert_1} \ : \ u \in L^1(\Gamma) \right\}.$$

Given $u \in L^1(\Gamma)$ and  $\mu \in  {\rm med}_\ell (u)$, we have
$$\frac{TV_\Gamma (u)}{ \Vert u - \mu \Vert_1} \leq \sup_{c \in \R} \frac{TV_\Gamma(u)}{\Vert u - c \Vert_1},$$
hence
$$h(\Gamma) = \lambda_1(\Gamma) \leq \alpha.$$
To prove the another inequality, let $D \subset \Gamma$, with $\ell(D) \leq \ell(\Gamma \setminus D)$, such that $$ h(\Gamma) = \frac{{\rm Per}_\Gamma(D)}{\ell(D)}.$$
take $v:= \1_D - \1_{\Gamma \setminus D}$. Then,
$$\alpha = \inf \left\{ \sup_{c \in \R} \frac{TV_\Gamma(u)}{\Vert u - c \Vert_1} \ : \ u \in L^1(\Gamma) \right\} = \inf \left\{ \max_{\vert c \vert \leq 1} \frac{TV_\Gamma(u)}{\Vert u - c \Vert_1} \ : \ u \in L^1(\Gamma) \right\}$$ $$\leq \max_{\vert c \vert \leq 1} \frac{TV_\Gamma(v)}{\Vert v - c \Vert_1} = \frac{2{\rm Per}_\Gamma(D)}{(1 - c)\ell(D)+ (1 + c) \ell(\Gamma \setminus D)} \leq \frac{{\rm Per}_\Gamma(D)}{\ell(D)} = h(\Gamma).$$
\end{proof}

For finite weighted graphs, it is well known that the first non--zero eigenvalue coincides with the Cheeger constant (see \cite{Chang1})
This result is not true for in infinite weighted graphs (see  \cite{MST1}). In the next result we will see that this is true  in metric graphs.
\begin{theorem} We have
\begin{equation}\label{equiegen}
h(\Gamma) =\inf \{\lambda\not= 0, \ \hbox{such that} \ \lambda \ \hbox{is a $\Gamma$-eigenvalue of  $-\Delta^\Gamma_1$} \}.
 \end{equation}
 Moreover, $h(\Gamma)$ is the first non-zero $\Gamma$-eigenvalue of $- \Delta^\Gamma_1$ and if $u$ is a  minimizer of problem \eqref{minnb}, then, there exists $t \geq 0$, such that $E_t(u)$ is a Cheeger cut of $\Gamma$ and $$\left(h(\Gamma), \frac{1}{\ell(E_t(u))} \1_{E_t(u)} \right)$$ is a $\Gamma$-eigenpair of $-\Delta^\Gamma_1$.
\end{theorem}

\begin{proof} By  Corollary~\ref{meddiaa}, if $(\lambda, u)$ is  an $\Gamma$-eigenpair of $-\Delta^\Gamma_1$ and $\lambda\not= 0$  then
$u \in \Pi(\Gamma)$. Now, $TV_\Gamma(u)=\lambda$,
thus,   as a consequence of  Theorem~\ref{varchar}, we have the
$$ h(\Gamma) \leq \lambda.$$

On the other hand, by Theorem \ref{varchar}, there exists $t \geq 0$, such that $E_t(u)$ is a Cheeger cut of $\Gamma$. Then,  $E_t(u)$ is $\Gamma$-calibrable. Hence, by Theorem \ref{igualconsts2},
$$h(\Gamma) = h_1^\Gamma(E_t(u)) = \sup \left\{\frac{1}{\Vert \z \Vert_\infty }  \ : \ \z \in X_K(\Gamma), \ \z' = \1_{E_t(u)}  \right\}.$$

Then, there exists a sequence $\z_n \in X_K(\Gamma)$ with $\z_n' = \1_{E_t(u)}$ for all $n \in \N$, such that $$h(\Gamma) = \lim_{n \to \infty} \frac{1}{\Vert \z_n \Vert_\infty }.$$
Now, since $h(\Gamma) >0$, we have $\{ \Vert \z_n \Vert_\infty \ : \ n \in \N\}$ is bounded. Thus, we can assume, taking a subsequence if necessary, that $$\z_n \to \z , \quad\hbox{weakly$^*$ in} \ L^\infty(\Gamma), \quad \hbox{and} \quad \z' = \1_{E_t(u)}.$$
 Let us see now that $\z \in X_K(\Gamma)$.
by Proposition \ref{megusta1}, we have that
\begin{equation}\lim_{n \to \infty} \int_\Gamma \z_n Du = \int_\Gamma \z Du, \quad \forall \, u \in BV(\Gamma).
\end{equation}
Fix $\v \in V(\Gamma)$. Applying Green's formula \eqref{Kintbpart} to $\z_n$ and $u \in BV(\Gamma)$ , we get
$$\int \z_n Du + \int_\Gamma u \z_n^{\prime} = \sum_{\v \in V(\Gamma)} \sum_{\e\in \E_\v(\Gamma)} [\z]_\e(\v) [u]_\e(\v).$$
Hence, taking $u$ such that $[u]_e(\v) =1$ for all $\e\in \E_\v(\Gamma)$ and $[u]_{\hat{\e}} =0$ if $\v \not\in \E_\v(\Gamma)$, we have
$$\int \z_n Du + \int_\Gamma u \z_n^{\prime} =  \sum_{\e\in \E_\v(\Gamma)} [\z_n]_\e(\v) [u]_\e(\v) =0.$$
Then, taking the limit as $n \to \infty$ and having in mind \eqref{intbpart}, we obtain
$$0 = \int \z Du + \int_\Gamma u \z^{\prime} =  \sum_{\e\in \E_\v(\Gamma)} [\z]_\e(\v) [u]_\e(\v).$$
Therefore, $\z \in X_K(\Gamma)$.

If we take $\tilde{\z}:= - h(\Gamma) \z$, and $v:= \frac{1}{\ell(E_t(u))} \1_{E_t(u)}$, we have $\tilde{\z} \in X_K(\Gamma)$,  $\Vert \tilde{\z} \Vert_\infty \leq 1$ and
$$- \tilde{\z}' = h(\Gamma)\1_{E_t(u)}, \quad TV_\Gamma(v) = h(\Gamma).$$
Therefore,
$$\left(h(\Gamma), \frac{1}{\ell(E_t(u))} \1_{E_t(u)} \right)$$ is a $\Gamma$-eigenpair of $-\Delta^\Gamma_1$.
\end{proof}

\begin{remark}\label{malRossi}{\rm In \cite[Theorem 1.3]{DelPR1} was proved that if we define
\begin{equation}\label{minJulio}\Lambda_{2,p}(\Gamma):= \inf \left\{ \frac{\int_\Gamma \vert u'(x) \vert^p dx}{\int_\Gamma \vert u(x) \vert^p dx} \ : \ u \in W^{1,p}(\Gamma), \ \int_\Gamma \vert u(x) \vert^{p-2}u(x) dx =0, u \not\equiv 0 \right\},\end{equation}
then if $u_p$ is a minimizer of \eqref{minJulio}, there exists a subsequence $p_j \to 1^+$, and $u \in BV(\Gamma)$, such that $$u_{p_j} \to u\quad \hbox{in} \ L^1(\Gamma),$$
being $u$ is a minimizer of \eqref{minnb}. Moreover,
$$\lim_{p \to 1^+} \Lambda_{2,p}(\Gamma) =  \Lambda_{2,1}(\Gamma),$$
where
$$\Lambda_{2,1}(\Gamma) := \inf \left\{ \vert Du \vert (\Gamma ) \ : \ u \in\Pi(\Gamma) \right\}. $$
Let us point out that, since for $u \in BV(\Gamma)$, in general,  $\vert Du \vert (\Gamma ) < TV_\Gamma(u)$, we have $\Lambda_{2,1}(\Gamma) < \lambda_1(\Gamma)$. Moreover, even more, with this definition of $\Lambda_{2,1}(\Gamma)$, it is possible that $\Lambda_{2,1}(\Gamma) =0$, for instance if $V(\Gamma) = \{\v_1, \v_2, \v_3 \}$ and $E(\Gamma) = \{ \e_1:=[\v_1, \v_2], \e_2:=[\v_2, \v_3] \}$, with $\ell_{\e_1} =\ell_{\e_2}$, then if $[u_i]_{\e_i} = \1_{(0, \ell_i)}$, we have $u_i \in\Pi(\Gamma)$ and $ \vert Du_i \vert (\Gamma ) =0$.
}$\blacksquare$
\end{remark}

Let $A \subset \Gamma$ with $ {\rm Per}_\Gamma(A) >0$, $\ell(A) = \frac12 \ell(\Gamma)$,  and $u= \frac{1}{\ell(\Gamma)} \left(\1_A - \1_{\Gamma\setminus A} \right)$. It is easy to see that
$TV_\Gamma(u)=\frac{2}{\ell(\Gamma)}{\rm Per}_\Gamma(A) = \frac{{\rm Per}_\Gamma(A)}{\ell(A)} >0$. Hence, since $\Vert u\Vert_1=1$ and $0 \in {\rm med}_\nu (u)$, we obtain the following result as a consequence of~Theorem \ref{varchar}.

 \begin{corollary}
Let $A \subset \Gamma$ with $\ell(A) = \frac12 \ell(\Gamma)$ and $u= \frac{1}{\ell(\Gamma)} \left(\1_A - \1_{\Gamma \setminus A} \right)$. Then,

 $h(\Gamma) = \frac{{\rm Per}_\Gamma(A)}{\ell(A)} \iff u= \frac{1}{\ell(\Gamma)} \left(\1_A - \1_{\Gamma \setminus A} \right) \ \hbox{ is a minimizer of } \ \eqref{minnb}$.
\end{corollary}

A similar result was proved in \cite[Theorem 1.4]{DelPR1}, but we have observed in Remark \ref{malRossi}  that their concept of perimeter is different to the one we used here.

In the next example we will see that there are Cheeger cup $E$ such that $\ell(E) < \frac12\ell(\Gamma)$.

\begin{example}{\rm Consider the metric graph $\Gamma$ with four vertices and three edges, $V(\Gamma) = \{\v_1, \v_2, \v_3, \v_4 \}$ and $E(\Gamma) = \{ \e_1:=[\v_1, \v_2], \e_2:=[\v_2, \v_3], \e_3:=[\v_2, \v_4] \}$.

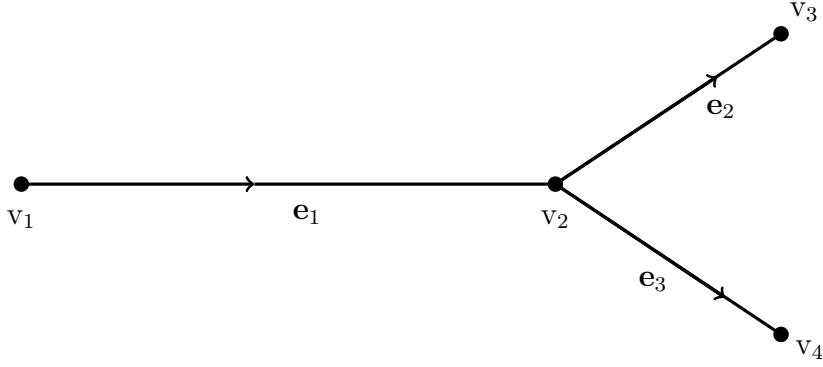
\begin{center}
\begin{tikzpicture}
 \tikzstyle{gordito2} = [line width=3]

\node (v1) at (-5.1,-1.2) {};
\node (v5) at (2,-1.2) {};

\node[below] at (v1) {$\v_1$};
\node[above] at (5.3093,1.0278) {$\v_3$};
\node[below] at (v5) {$\v_2$};
\node[above] at (5.3843,-3.4747) {$\v_4$};

\draw[gordito2]  (-5.1,-1) node (v2) {} circle (0.05);
\draw[gordito2]  (5,1) node (v3) {} circle (0.05);
\draw[gordito2]  (2,-1) node (v4) {} circle (0.05);
\draw[gordito2]  (5,-3) node (v4) {} circle (0.05);

\draw[->,line width=1.2]  (-5.1,-1)--(-2,-1);
\draw[line width=1.2]  (-2,-1) --(2,-1);
\draw[line width=1.2] (2,-1)--(5,-3);
\draw[->,line width=1.2] (2,-1)--(4.253,-2.5103);

\draw[line width=1.2](2,-1)--(5,1);

\draw[->,line width=1.2](2,-1)--(4.1524,0.4299);

\node at (-1.3,-1.4) {$\e_1$};
\node at (4.2,0) {$\e_2$};
\node at (3.3,-2.3) {$\e_3$};

\end{tikzpicture}
\end{center}

 If we assume that $\ell_{\e_i} = L$ for $i=1,2,3$, Then, each $\e_i$ is a Cheeger cup of $\Gamma$. In fact, if $D \subset \Gamma$ has ${\rm Per}_\Gamma(D) =1$, then $D \subset \e_i$. Now, if $D \not= \e_i$, then $$\frac{{\rm Per}_\Gamma(D)}{\ell(D)}  > \frac{{\rm Per}_\Gamma(\e_i)}{\ell(\e_i)} = \frac{1}{L}.$$ Moreover, if $D \subset \Gamma$ and $L < \ell (D) \leq \frac{3L}{2}$, then ${\rm Per}_\Gamma(D) \geq 2$. Hence
 $$\frac{{\rm Per}_\Gamma(D)}{\ell(D)} \geq \frac{2}{\frac{3L}{2}} = \frac{4}{3L} >\frac{1}{L}.$$
Thus
 $$h(\Gamma) = \frac{{\rm Per}_\Gamma(\e_i)}{\ell(\e_i)} = \frac{1}{L},$$
 and  consequently, each $\e_i$ is a Cheeger  cut of $\Gamma$.

 Moreover, $(h(\Gamma), \frac{1}{L} \1_{\e_i})$ is a $\Gamma$-eigenpair of $- \Delta^\Gamma_1$. For instance, for $\e_1$, if we define $\z$ as
 $$[\z]_{\e_1}(x):= - \frac{1}{L}x, \quad [\z]_{\e_i}(x):=  \frac{2}{L}x -2, \ i =2,3, \ x \in (0,L),$$
 then, $$- \z' = h(\Gamma) \1_{\e_1}\quad \hbox{in} \ \mathcal{D}(\Gamma), \quad \hbox{and} \quad TV_\Gamma(\frac{1}{L} \1_{\e_i}) = h(\Gamma).$$
 Therefore $(h(\Gamma), \frac{1}{L} \1_{\e_i})$ is an $\Gamma$-eigenpair of $- \Delta^\Gamma_1$.

 If we assume now that $\ell_{\e_1} > \ell_{\e_2} + \ell_{\e_3}$, then it is easy to see that $$h(\Gamma) = \frac{2}{\ell_{\e_1} + \ell_{\e_2} +\ell_{\e_3}}, \quad \hbox{and} \quad  \left(\v_1, c_{\e_1}^{-1}\left(\frac{\ell_{\e_1} + \ell_{\e_2} +\ell_{\e_3}}{2}\right)\right) \ \hbox{is a Cheeger  cut of $\Gamma$}.$$

}$\blacksquare$
\end{example}

Now we are going to get a Cheeger  inequality of type \eqref{Cheeger1} for metric graphs. For that let us introduce the Laplace operator $\Delta_\Gamma$ on the metric graph $\Gamma$.  This is a standard procedure
and we refer the interested reader to (\cite{Cattaneo}, \cite{BK}). For a function $u \in W^{1,1}(\Gamma)$, if $\e \in E(\Gamma)$ and $\v \in V(\Gamma)$, we define the normal derivative of $u$ at $\v$ as
$$\frac{\partial [u]_\e}{\partial n_\e}(\v) := \left\{ \begin{array}{ll} - [u]_\e (0+), \quad &\hbox{if} \ \ \v = \vi_\e \\[10pt]  [u]_\e (\ell-), \quad &\hbox{if} \ \ \v = \vf_\e.   \end{array} \right.$$

The operator $\Delta_\Gamma$ has domain
$$D(\Delta_\Gamma):= \left\{ u \in W^{2,2}(\Gamma) \ : \ u \ \hbox{continuous and } \  \sum_{\e \in E_\v(\Gamma)} \frac{\partial [u]_\e}{\partial n_\e}(\v) =0 \ \ \hbox{for all } \ \v \in V(\Gamma) \right\}$$
and it applies to any function $u \in D(\Delta_\Gamma)$  as follows
$$[\Delta_\Gamma u]_\e := ([u]_\e)_{xx}, \quad  \hbox{for all} \ \  \e \in E(\Gamma).$$

The energy functional form associated to $\Delta_\Gamma$ is given by
$$\mathcal{H}_\Gamma (u):= \int_\Gamma (u'(x))^2 dx = \sum_{\e \in E(\Gamma)} \Vert [u]_\e' \Vert^2_{L^2(0, \ell_\e)}.$$
We have
$$\mathcal{H}_\Gamma (u) = - \int_\Gamma u(x)\Delta_\Gamma u (x) dx, \quad \hbox{for} \ u \in D(\Delta_\Gamma).$$
The operator $- \Delta_\Gamma$ is selfadjoint in $L^2(\Gamma)$ and $$\sigma( -\Delta_\Gamma) \setminus \{0 \} = \{\mu_1(\Gamma), \mu_2(\Gamma), \ldots \},$$
being
$$\mu_1(\Gamma) = {\rm gap}(- \Delta_\Gamma)= \min\left\{ \frac{\mathcal{H}_\Gamma (u)}{\Vert u \Vert^2_{L^2(\Gamma)}} \ : \ u \in D(\mathcal{H}_\Gamma), \Vert u \Vert_{L^2(\Gamma)} \not=0, \ \int_\Gamma u(x) dx =0  \right\}.$$

\begin{theorem} We have the following {\it Cheeger Inequality}:
\begin{equation}\label{CheegerIn}
\frac14 h(\Gamma)^2 \leq {\rm gap}(- \Delta_\Gamma).
\end{equation}
\end{theorem}
\begin{proof}
 Let $u \in D(\mathcal{H}_\Gamma)$, with $\int_\Gamma u(x) dx =0$, such that
$$\frac{\mathcal{H}_\Gamma (u)}{\Vert u \Vert^2_{L^2(\Gamma)}} =  {\rm gap}(- \Delta_\Gamma).$$
If $\alpha \in {\rm med}_{\ell}(u)$ and $v:= u - \alpha$, we have $0 \in {\rm med}_{\ell}(v^2)$.  Then, by Theorem \ref{varchar}, we have
$$ h(\Gamma) \leq \frac{\displaystyle\int_\Gamma (v^2)'(x) dx }{\displaystyle\int_\Gamma v^2(x) dx}.$$
Now, since $\int_\Gamma u(x) dx =0$, we have
$$\int_\Gamma v^2(x) dx \geq \int_\Gamma u^2(x) dx.$$
On the other hand, by Cauchy- Schwartz
$$\int_\Gamma (v^2)'(x) dx = 2 \int_\Gamma v(x) v'(x) dx \leq 2 \left(\int_\Gamma v^2(x)dx \right)^{\frac12} \left( \int_\Gamma (v')^2(x)dx\right)^{\frac12}.$$
Thus
$$h(\Gamma)^2 \leq \frac{4 \left(\displaystyle \int_\Gamma v^2(x)dx \right) \left( \displaystyle\int_\Gamma (v')^2(x)dx\right)}{\left(\displaystyle \int_\Gamma v^2(x)dx \right)^2} = \frac{4  \displaystyle\int_\Gamma (u')^2(x)dx}{\displaystyle \int_\Gamma v^2(x)dx } \leq \frac{4\mathcal{H}_\Gamma (u)}{\Vert u \Vert^2_{L^2(\Gamma)}} = 4 {\rm gap}(- \Delta_\Gamma),$$
and therefore \eqref{CheegerIn} holds.
\end{proof}

Let us point out that the Cheeger Inequality \eqref{CheegerIn} was also prove by Nicaise \cite{Nicaise} (see also \cite{KN}and \cite{Post}) but with a different proof and for a different concept of perimeter.
 \bigskip

\noindent {\bf Acknowledgment.} The author have been partially supported  by the Spanish MCIU and FEDER, project PGC2018-094775-B-100 and by Conselleria d'Innovaci\'{o}, Universitats, Ci\`{e}ncia y Societat Digital, project AICO/2021/223.

\bigskip

Data sharing not applicable to this article as no datasets were generated or analysed during the current study.


\bigskip

\begin{thebibliography}{00}


 \bibitem{AM} N. Alon and V.D. Milman, {\it $\lambda_1$, Isoperimetric inequalities for graphs, and superconcentrators}.  J. Combin. Theory Ser. B {\bf 38}, (1985) 73--88.



\bibitem{ACCh} F. Alter, V. Caselles and A. Chambolle, {\it A characterization of convex calibrable sets in $\R^N$.}  Math. Ann. {\bf 332} (2005), 329--366.

\bibitem{AFP}
{L. Ambrosio, N. Fusco and D. Pallara, }
\newblock {Functions of Bounded Variation
and Free Discontinuity Problems}.
\newblock Oxford Mathematical Monographs, 2000.

\bibitem{ACMBook}
F.~Andreu, V.~Caselles, and J.M. Mazon, \textit{Parabolic Quasilinear Equations
Minimizing
 Linear Growth Functionals},   Progress in Mathematics, vol. 223, 2004. Birkhauser.


\bibitem{Anzellotti} G. Anzellotti,
\newblock {\it Pairings Between Measures and Bounded Functions
and Compensated Compactness},
\newblock Ann. di Matematica Pura ed Appl. IV (135)
(1983), 293-318.


 \bibitem{BCFK} G. Berkolaiko, R. Carlson, S. Fulling and P. Kuchment, {\it  Quantum Graphs and Their Applications.} Contemporary Mathematics, vol. 415. American Mathematical Society, Providence (2006).


\bibitem{BK} G. Berkolaiko and P. Kuchment.
	  Introduction to quantum graphs.
	Mathematical Surveys and Monographs, 186.
	American Mathematical Society, Providence, RI, 2013. xiv+270 pp.

\bibitem{BF} M. Bonforte and A. Figalli, {\it Total Variation Flow and Sign fast Diffusion in one dimension}. J. Differential Equations {\bf 252} (2012), 4455-4480.

 \bibitem{Brezis}
H. Brezis,
\newblock {Operateurs Maximaux Monotones}.
\newblock North Holland, Amsterdam, 1973.


\bibitem{BH1} T. B\"{u}hler and M. Hein, {\it  Spectral Clustering based on the graph $p$-Laplacian.} In Proceedings of the 26th
International Conference on Machine Learning, pp. 81--88. Omnipress, 2009.


 \bibitem{Cattaneo} C. Cattaneo, {\it The spectrum of the continuous Laplacian on a graph.} Monatsh. Math., {\bf 124}, (1997),
215--235.

\bibitem{Chang0} K. C. Chang, {\it Spectrum of the $1$-Laplacian operator}. Comm. Contemporary Math. {\bf 11} (2009), 865--894.

\bibitem{Chang1} K. C. Chang, {\it Spectrum of the $1$-Laplacian and Cheeger's Constant on Graphs}. Journal of Graph Theory {\bf 81} (2016), 167--207.


\bibitem{Chang1} K. C. Chang, {\it Spectrum of the $1$-Laplacian and Cheeger's Constant on Graphs}. Journal of Graph Theory {\bf 81} (2016), 167--207.

 \bibitem{Changetal01}
K. C. Chang, S. Shao and D. Zhang, {\it The 1-Laplacian Cheeger Cut: Theory and
Algorithms}   Journal of Computational Mathematics
{\bf 33} (2015), 443--467.


\bibitem{ChSZ} K.C. Chang, S. Shao and D. Zhang, {\it  Cheeger's cut, maxcut and the spectral theory of $1$-Laplacian on graphs.} Sci. China Math. {\bf 60} (2017), 1963--1980.


  \bibitem{Cheeger}   J. Cheeger, {\it A lower bound for the smallest eigenvalue of the Laplacian},
Problems in analysis: A symposium in honor of Salomon Bochner (1970), 195--199. Princeton Univ. Press.

\bibitem{Ch} F. Chung, {\it Spectral Graph Theory} (CBMS Regional Conference Series in Mathematics, No. 92), American Mathematical Society, 1997.


 \bibitem{Giorgi} E. De Giorgi, {\it Sulla propriet\`{a} isoperimetrica dell’ipersfera, nella classe degli
insiemi aventi frontiera orientata di misura finita}, Atti della Accademia
Nazionale dei Lincei. Mem. Cl. Sci. Fis. Mat. Nat. Sez. I {\bf 5} (1958), 33--44.

 \bibitem{GS} S. Gnutzmann and U. Smilansky, {\it Quantum graphs: applications to quantum
chaos and universal spectral statistics.} Adv. Phys. {\bf 55} (2006), 527--625.

 \bibitem{DelPR1} L. Del Pezzo and J. Rossi, {\it Clustering for Metric Graphs Using the $p$-Laplacian}. Michigan Math. J. {\bf 65} (2016), 451-472.


 \bibitem{DelPR2} L. Del Pezzo and J. Rossi, {\it The first eigenvalued of the   $p$-Laplacian on quatum graphs}. Anal. Math. Phys {\bf 6} (2016), 365-391.

     \bibitem{D} J. Dodziuk, {\it Difference equations, isoperimetric inequality and transience of certain random walks}.  Trans. Amer. Math. Soc. {\bf  284} (1984), 787--794.


\bibitem{EkelandTemam} I. Ekeland, R. Temam, {\it Convex analysis and variational problems}. North-Holland Publ. Company, Amsterdam, 1976.

  \bibitem{EKKST}  P. Exner, J.P.  Keating, P. Kuchment, T. Sunada, and A. Teplyaev, {\it  Analysis on graphs and its applications.} In: Proceedings of Symposia in Pure Mathematics,
vol. 77, Providence, RI. Am. Math. Soc. (2008)


\bibitem{FF1} L.R. Jr. Ford and D.R. Fulkerson, {\it  Maximal flow through a network.} Canad. J. Math. {\bf 8} (1956), 399--404.



\bibitem{FK} V. Fridman, B. Kawohl, {\it Isoperimetric estimates for the first eigenvalue of the p-Laplace operator and the
Cheeger constant}. Comment. Math. Univ. Carolinae {\bf 44} (2003), 659--667.


 \bibitem{Grieser} D, Grieser, {\it  The first eigenvalue of the Laplacian, isoperimetric constants, and the max flow min cut theorem.} Archiv der Mathematik



    \bibitem{HB} M. Hein and T. B\"uhler, {\it   An Inverse Power Method for Nonlinear
Eigenproblems with Applications in
$1$-Spectral Clustering and Sparse PCA}. Advances in Neural Informatio Proceessing Systems {\bf 23} (2010), 847--855.


 \bibitem{KF} B. Kawohl and V. Fridman, {\it Isoperimetric estimates for the first eigenvalue of the $p$-Laplace operator and the Cheeger constant.} Comment. Math. Univ. Carolin.
{\bf 44} (2003), 659--667.


   \bibitem{KM} J.B. Kennedy and D. Mugnolo, {\it The Cheeger constant of a quantum graph.} Proc. Appl. Math. Mech. {\bf 16} (2016), 875--876.

          \bibitem{KN} A. Kostenko and N. Nicolusi, {\it Spectral estimates for the infinite quatum graphs}. Cal. Var. Partial Differential Equations {\bf 58} (2019),, no. 1, Paper 15, 40 pp.

  \bibitem{KS}     V. Kostrykin and R. Schrader, {\it  Laplacians on metric graphs: eigenvalues, resolvents and semigroups.} In
Quantum graphs and their applications, volume 415 of Contemp. Math., pages 201–225. Amer. Math. Soc.,
Providence, RI, 2006.

\bibitem{Leonardi} G.P. Leonardi. {\it An overview on the Cheeger problem.} In New Trends in Shape Optimization, volume 166 of Internat. Ser. Numer. Math., pages 117–139. Springer Int. Publ., 2015.


\bibitem{Luxburg} U. von Luxburg, {\it A tutorial on spectral clustering.} Stat. Comput. {\bf 17} (2007), 395--416.

\bibitem{Mazon} J.M. Maz\'{o}n, {\it The Total Variation Flow in Metric Graphs}. Mathematics in Engineering, 2023, 5(1): 1-38. doi: 10.3934/mine.2023009.

\bibitem{MRT}  J. M. Maz\'{o}n, J. D. Rossi and J. Toledo, {\it Nonlocal Perimeter, Curvature and Minimal Surfaces for Measurable Sets}.  J. Anal. Math. {\bf 138} (2019), no. 1, 235--279.

\bibitem{MRTLibro}  J. M. Maz\'{o}n, J. D. Rossi and J. Toledo, {Nonlocal Perimeter, Curvature and Minimal Surfaces for Measurable Sets}.  Frontiers in Mathematics, Birkh\"auser, 2019.



\bibitem{MST0}  J. M. Maz\'{o}n, M. Solera and J. Toledo, {\it   The heat flow on metric random walk spaces}. J. Math. Anal. Appl. 483, 123645 (2020).

\bibitem{MST1}  J. M. Maz\'{o}n, M. Solera and J. Toledo, {\it The total variation flow in  metric random walk spaces}. Calc.Var. (2020) 59:29.

     \bibitem{Meyer} Y. Meyer, {\it Oscillating patterns in image processing and nonlinear evolution equations,} University lecture Series, 22. American Mathematical Society, Providance, RI, 2001.

   \bibitem{Mugnolo}  D. Mugnolo, {\it Semigroup methods for evolution equations on networks. } Understanding Complex Systems. Springer, Cham, 2014.

       \bibitem{Nicaise}    S. Nicaise, {\it  Spectre des r\'{e}seaux topologiques finis}. Bull. Sci. Math., II. S\'{e}r., 111:401–413, 1987.

       \bibitem{Parini} E. Parini, {\it An Introduction to the Cheeger Problem}. Surveys in Math Appl. {\bf 6} (2011), 9-22.

      \bibitem{Post} O.Post {\it Spectrala nalysis of metricgraphs and related spaces}.In: Arzhantseva,G.,Valette,A.(eds.)Limits of Graphs in Group Theory and Computer Science,pp.109–140.Presses Polytechniqueset Universitaires Romandes, Lausanne (2009)

\bibitem{Strang1} G. Strang, {\it  Maximal flow through a domain.} Math. Program. {\bf 26} (1983), 123--143.
    \bibitem{Strang2} G. Strang, {\it  Maximum flow and minimaum cuts in the plane.} J. Glob. Optim.  {\bf 47} (2010), 527--535.

  \bibitem{SB1}  A. Szlam and X. Bresson,  {\it  Total Variation and Cheeger Cuts}.  Proceedings of the 27 th International Confer-
ence on Machine Learning, Haifa, Israel, 2010.

  \bibitem{Yau} S-T. Yau, {\it Isoperimetric constants and the first eigenvalue of a compact Riemann Manifold}. Ann. scient. Ec. Norm. Sup., 4ª s\'{e}rie, t.8, (1975), 487-507.

\end{thebibliography}
 \end{document}